\documentclass[11pt,UTF8]{article}
\usepackage{bbm}
\usepackage[all,cmtip]{xy}
\usepackage{amsfonts,amssymb,amsmath,amscd,amsthm,mathtools}
\usepackage{mathrsfs}
\usepackage{latexsym}
\usepackage{color}
\usepackage{verbatim,enumitem}
\usepackage{hyperref}
\input diagxy
\usepackage{tikz-cd}
\usepackage[total={155mm,225mm}]{geometry}

\DeclareMathOperator{\with}{\&}

\DeclareMathOperator{\thda}{{\rotatebox[origin=c]{-90}{$\twoheadrightarrow$}}}

\theoremstyle{plain}
\newtheorem{thm}{Theorem}[section]
\newtheorem{lem}[thm]{Lemma}
\newtheorem{prop}[thm]{Proposition}
\newtheorem{cor}[thm]{Corollary}
\theoremstyle{definition}
\newtheorem{defn}[thm]{Definition}
\newtheorem{ques}[thm]{Question}

\newtheorem{exmp}[thm]{Example}
\newtheorem{rem}[thm]{Remark}

\newtheorem*{con}{Convention}

\begin{document}
	
\def\oto{{\bfig\morphism<180,0>[\mkern-4mu`\mkern-4mu;]\place(86,0)[\circ]\efig}}
\def\rto{{\bfig\morphism<180,0>[\mkern-4mu`\mkern-4mu;]\place(78,0)[\mapstochar]\efig}}
	
\newcommand{\lam}{\lambda}
	\newcommand{\da}{\downarrow}
	\newcommand{\Da}{\Downarrow\!}
	\newcommand{\D}{\Delta}
	\newcommand{\ua}{\uparrow}
	\newcommand{\ra}{\rightarrow}
	\newcommand{\la}{\leftarrow}
	\newcommand{\lra}{\longrightarrow}
	\newcommand{\lla}{\longleftarrow}
	\newcommand{\up}{\upsilon}
	\newcommand{\Up}{\Upsilon}
	\newcommand{\ep}{\epsilon}
	\newcommand{\ga}{\gamma}
	\newcommand{\Ga}{\Gamma}
	\newcommand{\Lam}{\Lambda}
	\newcommand{\CF}{{\cal F}}
	\newcommand{\CG}{{\cal G}}
	\newcommand{\CH}{{\cal H}}
	\newcommand{\CN}{{\mathcal{N}}}
	\newcommand{\CB}{{\cal B}}
	\newcommand{\CI}{{\cal I}}
	\newcommand{\CT}{{\cal T}}
	\newcommand{\CS}{{\cal S}}
	\newcommand{\CP}{{\cal P}}
	\newcommand{\CQ}{\mathcal{Q}}
	
	\newcommand{\Om}{\Omega}
	\newcommand{\bv}{\bigvee}
	\newcommand{\bw}{\bigwedge}
	\newcommand{\dda}{\downdownarrows}
	\newcommand{\dia}{\otimessuit}
	\newcommand{\y}{\mathfrak{y}}
	\newcommand{\id}{{\rm id}}
	\newcommand{\sub}{{\rm sub}}
\newcommand{\sQ}{{\sf Q}}
\newcommand{\SFQ}{\sQ\text{-}{\sf SemFil}}
\newcommand{\CSFQ}{{\sf ConSF}}
\newcommand{\FQ}{\sQ\text{-}{\sf Fil}}
\newcommand{\CFQ}{{\sf ConFil}}
\newcommand{\SPF}{{\sf SPF}}
\newcommand{\BSF}{{\sf BSP}}
\newcommand{\BCF}{{\sf ConBSF}}
\newcommand{\topF}{\top\text{-}{\sf Fil}}
\newcommand{\sfm}{{\sf m}}
\newcommand{\sfe}{{\sf e}}
\newcommand{\sfn}{{\sf n}}
\newcommand{\sfd}{{\sf d}}		
	\numberwithin{equation}{section}
	\renewcommand{\theequation}{\thesection.\alph{equation}}
\newcommand{\pleq}{\sqsubseteq}
\newcommand{\QCat}{\sQ\text{-}{\sf Cat}}
\allowdisplaybreaks

\title{The  saturated prefilter monad
}
 \author{Hongliang Lai, Dexue Zhang, Gao Zhang \\ {\small School of Mathematics, Sichuan University, Chengdu 610064, China}\\  {\small hllai@scu.edu.cn, dxzhang@scu.edu.cn, gaozhang0810@hotmail.com } }

 \date{}

\maketitle

\begin{abstract} This paper considers some extensions of the notion of filter to the quantale-valued context, including saturated prefilter, $\top$-filter and bounded saturated prefilter. The question is
whether these constructions give rise to monads on the category of sets. It is shown that the answer depends on the structure of the quantale. Specifically,  if the quantale is the unit interval equipped with a continuous t-norm, then these constructions give rise to monads if and only if the implication operator corresponding to that t-norm is continuous at each point off the diagonal.



\end{abstract}


\section{Introduction}

The filter monad on the category of sets  plays a crucial role in topology and order theory. The multiplication of the filter monad helps us to express iterative limits  in  topology; the Eilenberg-Moore algebras of the   filter  monad are the continuous lattices and hence the injective $T_0$ spaces; the Eilenberg-Moore algebras of the ultrafilter monad   are   the compact  Hausdorff spaces; and etc. There exists a large number of works that are related to applications of the filter monad in  topology  and order theory, for example,  the monographs \cite{Gierz2003,Monoidal top,Manes1976} and the articles \cite{Barr,Day,Escardo97,Escardo98,Manes1974,Manes2010,Wyler}.

The notion of filter  has been extended to the enriched (= quantale-valued in this paper) setting in different ways, resulting in   prefilters \cite{Lowen1979,Lowen82},   $\top$-filters \cite{Hoehle1982,Ho99,Ho2001},  functional ideals \cite{Lowen15,LOV2008,LV2008} (for Lawvere's quantale), and (various kinds of)  $\sQ$-filters  (or, fuzzy filters) \cite{EG92,Ga95,Ho99,Ho2001}. $\sQ$-filters give rise to a monad on the category of sets, it is indeed a submonad of the double $\sQ$-powerset monad,  see   \cite{Ga95,Ho2001}. It is shown in \cite{CVO} that the functional ideals   also give  rise to a monad on the category of sets. But, it is  still unknown whether so do  the prefilters.

In this paper, we consider the question whether saturated prefilters   \cite{Hoehle1982,Lowen82}, a special kind of prefilters, give rise to a monad. Since a monad consists of a functor and two natural transformations, the question will be stated precisely  in Section 3. The answer depends on the structure of the quantale.  In the case that  the quantale  $\sQ$ is the interval $[0,1]$ equipped with a continuous t-norm, a necessary and sufficient condition  is presented in Section 5. The key idea is to identify the  saturated prefilter functor with a subfunctor of the $\sQ$-semifilter functor, which is  a submonad of the double $\sQ$-powerset monad. In Section 6 and Section 7, the same technique is applied to study the  $\top$-filter functor and the bounded saturated prefilter functor,  respectively.

\section{Preliminaries}

A complete lattice $L$ is \emph{meet continuous} if for all $p\in L$, the map $p\wedge- : L\lra L$ preserves directed joins; that is,  $$p\wedge \bv D=\bv_{d\in D}(p\wedge d)$$ for each directed set $D\subseteq L$.

In a complete lattice $L$, $x$ is way below $y$,  in symbols  $x\ll y$, if for every directed set $D\subseteq L$, $y\leq \bv D$ always implies that $x\leq d$ for some $d\in D$. It is clear that for all $p\in L$, $\thda  p=\{x\in L\mid  x\ll p\}$ is a directed subset of $L$. A complete lattice $L$ is \emph{continuous} if  $p=\bv\thda p$ for all $p\in L$.  It is known that each continuous lattice is meet continuous  \cite{Gierz2003}.

In this paper, by a quantale we mean a \emph{commutative  unital  quantale}  in the sense of \cite{Rosenthal1990}. In the language of category theory, a quantale is a small, complete and symmetric monoidal closed category \cite{Borceux1994,Kelly}. Explicitly, a  quantale
\[\sQ=(\sQ,\with,k)\]
is a commutative monoid with $k$ being the unit, such that the underlying set $\sQ$ is a complete lattice with a bottom element $0$ and a top element $1$,   and that the multiplication $\with$ distributes over arbitrary joins. The  multiplication $\&$ determines a binary operator $\ra$,   sometimes called the implication operator of $\&$, via the adjoint property:
\[p\with q\leq r\iff q\leq p\ra r.\]

Given a   quantale \((\sQ,\with,k)\), we say that
\begin{itemize}   \item $\sQ$ is \emph{integral}, if the unit $k$ coincides with the top element of the complete lattice $\sQ$;
\item $\sQ$ is \emph{meet continuous}, if the complete lattice $\sQ$ is meet continuous; and
\item $\sQ$ is \emph{continuous}, if the complete lattice $\sQ$ is   continuous. \end{itemize}

Lawvere's quantale $([0,\infty]^{\rm op},+,0)$  \cite{Lawvere1973} is integral and continuous. Quantales obtained by endowing the unit interval $[0,1]$ with a continuous t-norm are of particular interest in this paper.
A continuous t-norm  \cite{Klement2000} is, actually,  a continuous map $\with:[0,1]^2\lra[0,1]$  that makes $([0,1],\with ,1)$ into a quantale. Given a continuous t-norm $\with $, the quantale $\sQ=([0,1],\&,1)$  is integral and continuous. The way below relation in $[0,1]$ is  as follows:    $x\ll y$  if either $x=0$ or $x<y$. Such quantales play a decisive role in the BL-logic of H\'{a}jek \cite{Ha98}.

\begin{exmp}Some basic   continuous t-norms and their implication operators:
\begin{enumerate}[label={\rm(\arabic*)}] 
\item  The G\"{o}del t-norm: \[ x\with y= \min\{x,y\}; \quad x\ra y=\begin{cases}
		1,&x\leq y,\\
		y,&x>y.
		\end{cases} \] The implication  operator $\ra$ of the G\"{o}del t-norm is  continuous  except at   $(x,x)$, $x<1$.

\item  The product t-norm:  \[ x\with_P y=xy; \quad x\ra y=\begin{cases}
		1,&x\leq y,\\
		y/x,&x>y.
		\end{cases} \] The implication operator  $\ra$
of the product t-norm is  continuous  except at   $(0,0)$. The quantale $([0,1], \with_P, 1)$ is  isomorphic to Lawvere's quantale $([0,\infty]^{\rm op},+,0)$.

\item  The {\L}ukasiewicz t-norm:\[ x\with_{\L} y=\max\{0,x+y-1\}; \quad x\ra y=\min\{1-x+y,1\}. \] The implication operator  $\ra$
    of the {\L}ukasiewicz t-norm
    is   continuous on $[0,1]^2$.
	\end{enumerate}
\end{exmp}

Let $\with $ be a continuous t-norm. An element $p\in [0,1]$ is   \emph{idempotent}  if $p\with p=p$.

\begin{prop}{\rm(\cite[Proposition 2.3]{Klement2000})} \label{idempotent}
Let $\&$ be a continuous t-norm on $[0,1]$ and $p$ be an idempotent element of $\&$. Then $x\with y= \min\{x, y\} $ whenever $x\leq p\leq y$.
 \end{prop} It follows immediately  that $y\ra x=x$ whenever  $x< p\leq y$ for some idempotent   $p$. Another consequence of  Proposition \ref{idempotent} is that for any idempotent elements $p, q$    with $p<q$,  the restriction of $\with $ to $[p,q]$, which is also denoted by $\with$,  makes $[p,q]$ into a commutative quantale with $q$ being the unit element.    The following  theorem, known as the \emph{ordinal sum decomposition theorem},  plays a prominent role in the theory of continuous t-norms.

\begin{thm} {\rm(\cite{Klement2000,Mostert1957})}
\label{ordinal sum} Let $\with $ be a continuous t-norm. If $a\in [0,1]$ is non-idempotent, then there exist idempotent elements $a^{-}, a^{+}\in [0,1]$ such that $a^-<a<a^+$ and that the quantale  $([a^{-},a^{+}],\with ,a^{+})$ is either isomorphic to  $([0,1],\with_{\L},1)$ or to $([0,1],\with_{P},1)$. Conversely, for each   set of disjoint open intervals $\{(a_n,b_n)\}_n$ of $[0,1]$, the binary operator \[x\with y\coloneqq\begin{cases}a_n+(b_n-a_n)T_n\Big(\displaystyle{\frac{x-a_n}{b_n-a_n}, \frac{y-a_n}{b_n-a_n}}\Big), & (x,y)\in [a_n,b_n]^2,\\
\min\{x,y\}, & {\rm otherwise} \end{cases}\] is a continuous t-norm, where each $T_n$ is a continuous t-norm on $[0,1]$. \end{thm}

Let \(\sQ=(\sQ,\with,k)\) be a  quantale.  A \emph{$\sQ$-category} \cite{Monoidal top,Lawvere1973,Wagner97} (a.k.a $\sQ$-ordered set, see e.g. \cite{Belo2004,LZ2020}) consists of a set $A$ and a map $a:A\times A\to \sQ$  such that
 \[k\leq a(x,x)\quad \mbox{and}\quad
a(y,z)\with  a(x,y)\leq a(x,z)\]
 for all $x,y,z\in A$. It is customary to write $A$ for the pair $(A, a)$ and $A(x,y)$ for $a(x,y)$ if no confusion would arise.

A \emph{$\sQ$-functor} $f: A\to B$ between $\sQ$-categories is a map $f:A\to B$ such that $$A(x,y)\leq B(f(x),f(y))$$
for all $x,y\in A$.

\begin{exmp}   \label{d_L} For all $p,q\in\sQ$, let \[d_L(p,q)=p\ra q.\] Then $(\sQ,d_L)$ is a   $\sQ$-category. Generally, for each set $X$, $(\sQ^X,\sub_X)$ is a  $\sQ$-category, where for all $\lam,\mu\in\sQ^X$,  \[ \sub_X(\lam,\mu)=\bigwedge_{x\in X}\lam(x)\ra\mu(x).   \]

For each map $f: X\lra Y$   and each $\lam\in\sQ^X$, define $f(\lam)\in\sQ^Y$ by  \[f(\lam)(y)=\bv_{f(x)=y} \lam(x).\]Then,  $f:(\sQ^X,\sub_X)\lra(\sQ^Y,\sub_Y)$ and  $-\circ f:(\sQ^Y,\sub_Y)\lra(\sQ^X,\sub_X)$ are $\sQ$-functors such that  \[\sub_Y(f(\lam),\mu)= \sub_X(\lam,\mu\circ f)\] for all $\lam\in\sQ^X$ and $\mu\in\sQ^Y$. This  fact is an instance of enriched Kan extensions  \cite{Borceux1994,Kelly,Lawvere1973}. \end{exmp}

\section{The question}

A proper filter on a set $X$ is an upper set of $(\CP(X),\subseteq)$  that is closed under finite meets and does not contain the empty set $\emptyset$.
The notion of filter has been extended to the quantale-valued setting in different ways: prefilter  and $\sQ$-semifilter.

\begin{defn} (Lowen, \cite{Lowen1979}) A prefilter $F$ on a   set $X$ is a   subset of $\sQ^X$   and such that \begin{enumerate}[label=(\roman*)] \item $k_X\in F$, where $k_X$ is the constant map $X\lra\sQ$ with value $k$; \item if $\lam,\mu\in F$ then $\lam\wedge\mu\in F$; \item if $\lam\in F$ and $\lam\leq\mu$  then $\mu\in F$. \end{enumerate}\end{defn}

It should be warned that in the above definition, a prefilter $F$ is allowed to contain the constant map $0_X$, in which case $F=\sQ^X$; but, in \cite{Lowen1979} a prefilter is required not to contain $0_X$.

\begin{defn} (H\"{o}hle, \cite[Definition 1.5]{Hoehle1982}) Let $F$ be a prefilter   on a  set $X$. Then we say  that \begin{enumerate}[label=(\arabic*)] \item
$F$ is saturated if $\lam\in F$ whenever $\bv_{\mu\in F}\sub_X(\mu,\lam)\geq k$.  \item $F$ is a $\top$-prefilter if it is saturated and  $\bv_{x\in X}\lam(x)\geq k$ for all $\lam\in F$. \end{enumerate}
\end{defn}

Saturated prefilters on $X$ are closed with respect to intersection. For each prefilter $F$, the smallest saturated prefilter containing $F$ is called the \emph{saturation} of $F$.

For each set $X$, let $\SPF(X)$  denote the set of all saturated prefilters  on $X$. For each map $f:X\lra Y$ and each saturated prefilter $F$ on $X$, let \[f(F)=\{\lam\in\sQ^Y\mid \lam\circ f\in F\}.\] Then $f(F)$ is a saturated prefilter on $Y$. That $f(F)$ is a prefilter is clear, to see that it is saturated, suppose that\[\bv_{\lam\circ f\in F} \sub_Y(\lam,\mu) \geq k.\] Then \[\bv_{\lam\circ f\in F}\sub_X(\lam\circ f,\mu\circ f)\geq \bv_{\lam\circ f\in F}\sub_Y(\lam,\mu)\geq k,\] which implies that $\mu\circ f\in F$, hence $\mu\in f(F)$. In this way,  we obtain a functor \[\SPF:{\sf Set}\lra{\sf Set}.\]

For each map $f:X\lra Y$ and each $\top$-filter $F$ on $X$, $f(F)$ is clearly a $\top$-filter on $Y$. So, assigning to each set $X$   the set $\topF(X)$ of all $\top$-filters on $X$ defines a functor \[\topF:{\sf Set}\lra{\sf Set},\] which is a subfunctor of $\SPF$.

Let $X$ be a set. For each $x\in X$, let  \begin{equation}\label{defn of d}\mathfrak{d}_X(x) =\{\lam\in\sQ^X\mid \lam(x)\geq k\};\end{equation}  for each saturated prefilter $\CF$ on $\SPF(X)$, let
  \begin{equation}\label{defn of n}\mathfrak{n}_X(\CF)= \{\lam\in\sQ^X\mid  \widetilde{\lam} \in\CF\},\end{equation} where,  \[\widetilde{\lam}(F)=\bv_{\mu\in F}\sub_X(\mu,\lam)  \] for every saturated prefilter $F$ on $X$.

As we shall see in Proposition \ref{natural of d and n}, $\mathfrak{d}=\{\mathfrak{d}_X\}_X$ is a natural transformation   ${\rm id}\lra\SPF$  and $\mathfrak{n}=\{\mathfrak{n}_X\}_X$ is a natural transformation   $\SPF^2\lra\SPF$. The  formulas (\ref{defn of d}) and (\ref{defn of n}) appeared  in Yue and Fang \cite{YF2020} for   $\top$-filters (see Remark \ref{6.4} below). Proposition \ref{natural of d and n} shows that these formulas are obtained in a  \emph{natural} way.
Now we   state the main question of this paper.

\begin{ques}
When is the triple $(\SPF, \mathfrak{n},\mathfrak{d})$    a monad?\end{ques}

\section{The  $\sQ$-semifilter monad}
The following definition is a slight modification of that of  \emph{$\sQ$-filter}  in \cite{Ga95,Ho99,HS99}.
\begin{defn} \label{Q-filter} A $\sQ$-semifilter   on a  set $X$ is a  map $\mathfrak{F}: \sQ^X\lra \sQ$ subject to the following conditions:  for all    $\lam,\mu\in \sQ^X$, \begin{enumerate}[label=(F\arabic*)]
\item \label{FF1} $\mathfrak{F}(k_X)\geq k$; \item \label{FF2} $\mathfrak{F}(\lam)\wedge\mathfrak{F}(\mu)\leq\mathfrak{F}(\lam\wedge\mu)$; \item \label{FF3} $\sub_X(\lam,\mu)\leq\mathfrak{F}(\lam)\ra \mathfrak{F}(\mu)$; in other words, $\mathfrak{F}:(\sQ^X,\sub_X)\lra(\sQ,d_L)$ is a $\sQ$-functor. \end{enumerate}

A $\sQ$-semifilter   $\mathfrak{F}$  is called a $\sQ$-filter if it satisfies moreover \begin{enumerate} \item[(F4)] $\mathfrak{F}(p_X)\leq p$  for all $p\in\sQ$, where $p_X$ is the constant map $X\lra\sQ$ with value $p$.\end{enumerate}\end{defn}

Condition (F3) is equivalent to  \begin{enumerate} \item[(F3')] $\mathfrak{F}$ preserves order and $p\with\mathfrak{F}(\lam)\leq \mathfrak{F}(p\with\lam)$ for all $p\in\sQ$ and $\lam\in\sQ^X$. \end{enumerate}  And, (F3)  implies that the inequalities in (F2) and (F4) are actually  equalities.

For each set $X$, let $\SFQ(X)$   be the set of all $\sQ$-semifilters on $X$. For each map $f:X\lra Y$ and each $\mathfrak{F}\in\SFQ(X)$, define \[f(\mathfrak{F}):\sQ^Y\lra\sQ\]   by \[f(\mathfrak{F})(\mu)=\mathfrak{F}(\mu\circ f).\] Then $f(\mathfrak{F})$ is a $\sQ$-semifilter on $Y$. Therefore, we obtain a functor \[\SFQ:{\sf Set}\lra{\sf Set}.\]
Likewise, assigning to each set $X$ the set $\FQ(X)$   of all $\sQ$-filters on $X$ gives a functor \[\FQ:{\sf Set}\lra{\sf Set},\] which is a subfunctor of $\SFQ$.

For each set $X$ and each $x$ of $X$, the map \[\sfe_X(x):\sQ^X\lra\sQ, \quad \sfe_X(x)(\lam)=\lam(x)\]  is  a $\sQ$-filter. And, \[\sfe=\{\sfe_X\}_X\] is a natural transformation   ${\rm id}\lra\SFQ$.

For each $\sQ$-semifilter $\mathbb{F}$ on   $\SFQ(X)$, define \[\sfm_X(\mathbb{F}):\sQ^X\lra\sQ\] by \[\sfm_X(\mathbb{F})(\lam)=\mathbb{F}(\widehat{\lam}),\] where $\widehat{\lam}:\SFQ(X)\lra\sQ$ is   given by \[\widehat{\lam}(\mathfrak{G})= \mathfrak{G}(\lam). \] Then $\sfm_X(\mathbb{F})$ is a $\sQ$-semifilter on $X$. That $\sfm_X(\mathbb{F})$ satisfies \ref{FF1} and \ref{FF2} is obvious. To see that it satisfies \ref{FF3}, let $\lam,\mu\in\sQ^X$. Then \begin{align*}\sub_X(\lam,\mu)&\leq\bw_{\mathfrak{G}\in\SFQ(X)}(\mathfrak{G}(\lam)\ra \mathfrak{G}(\mu))\\ &=\sub_{\SFQ(X)}(\widehat{\lam},\widehat{\mu})\\ &\leq \mathbb{F}(\widehat{\lam})\ra\mathbb{F}(\widehat{\mu}) \quad~~(\mathbb{F}~\text{is a $\sQ$-semifilter})\\ &= \sfm_X(\mathbb{F})(\lam)\ra\sfm_X(\mathbb{F})(\mu).\end{align*} The $\sQ$-semifilter $\sfm_X(\mathbb{F})$ is called the \emph{diagonal $\sQ$-semifilter} (or the \emph{Kowalsky sum})  of $\mathbb{F}$.


\begin{prop}{\rm(C.f. \cite[Theorem 2.4.2.2]{Ho2001})} \label{Q-semifilter monad} The class of maps  \[\sfm=\{\sfm_X\}_X\] is a natural transformation   $\SFQ^2\lra\SFQ$ and  the triple $(\SFQ,\sfm,\sfe)$ is a monad on the category of sets. \end{prop}
\begin{proof}It is known that the filter monad arises from an
adjunction between the categories of sets and  semilattices, see e.g. Examples II.3.1.1\thinspace(5) in \cite{Monoidal top}. Likewise,  $(\SFQ,\sfm,\sfe)$   arises from  a natural adjunction between the category of sets and a category of certain $\sQ$-categories. 
We include below a proof of the conclusion   by displaying  $(\SFQ,\sfm,\sfe)$ as a \emph{submonad of  the double $\sQ$-powerset monad}, since the idea will be applied to some other functors in this paper. \end{proof}

Let $(S, \mu,\eta)$ be a monad on the category of sets; and let $T$ be a subfunctor of $S$. Suppose that $T$ satisfies the following conditions: \begin{enumerate}[label=(\roman*)] \item for each set $X$ and each $x\in X$, $\eta_X(x)\in T(X)$. So $\eta$ is also a natural transformation from the identity  functor to $T$. \item $T$ is closed under multiplication in the sense that for each set $X$ and each $\mathbb{H}\in TT(X)$, \[\mu_X\circ(\mathfrak{i}*\mathfrak{i})_X(\mathbb{H})\in T(X),\] where $\mathfrak{i}$ is the inclusion transformation of $T$ to $ S$ and $\mathfrak{i}*\mathfrak{i}$ stands for the horizontal composite of $\mathfrak{i}$ with itself.
So  $\mu$ determines a natural transformation $T^2\lra T$, which is also denoted by $\mu$. \end{enumerate} Then,   $(T,\mu,\eta)$ is a monad and the inclusion transformation $\mathfrak{i}:T\lra S$ is a monad morphism. We call $(T,\mu,\eta)$ (or  the functor $T$, for simplicity) a submonad of $(S,\mu,\eta)$ \cite{Manes2003,Manes2010}.

Given a set $B$,   the contravariant functor \[\CP_B:{\sf Set}^{\rm op}\lra{\sf Set}, \] which sends each set $X$ to (the $B$-powerset) $B^X$, is    right adjoint to its opposite \[\CP_B^{\rm op}:{\sf Set}\lra{\sf Set}^{\rm op}.\] The   monad $(\mathbf{P}, \sfm,\sfe)$ arising from this adjunction is called the \emph{double $B$-powerset monad} \cite[Remark 1.2.7]{Ho2001}. Explicitly, \[\mathbf{P}(X)=\CP_B\circ \CP_B^{\rm op}(X)=B^{B^X};\] the unit $\sfe_X:X\lra \mathbf{P}(X)$ is given by \[\sfe_X(x)(\lam)=\lam(x) \]for all $x\in X$  and $\lam\in B^X$;  and the multiplication $\sfm_X:\mathbf{P}\mathbf{P}(X)\lra\mathbf{P}(X)$ is given by \[\sfm_X(\mathbb{H})(\lam)= \mathbb{H}(\widehat{\lam}) \] for all $\mathbb{H}: B^{\mathbf{P}(X)}\lra B$ and $\lam\in B^X$ with $\widehat{\lam}:\mathbf{P}(X)\lra B$   given by \[\widehat{\lam}(\mathfrak{A})= \mathfrak{A}(\lam)\] for all $\mathfrak{A}\in  \mathbf{P}(X)$.

By the paragraph before Proposition \ref{Q-semifilter monad}, one immediately sees that $ \SFQ$  is  a submonad  of   $(\mathbf{P}, \sfm,\sfe)$ and    $\FQ$ is a submonad of $(\SFQ,\sfm,\sfe)$.

\begin{rem} (C.f. Examples II.3.1.1  in \cite{Monoidal top})
The construction of the submonads \[(\SFQ,\sfm,\sfe)\quad \text{and}\quad  (\FQ,\sfm,\sfe)\] is typical. When the set $B$ comes with some structures, we may be able to formulate some submonads of the double $B$-powerset monad \[(\mathbf{P}, \sfm,\sfe)\]  by aid of the structures on $B$. The (proper) filter monad and the ultrafilter monad are   examples of this construction. To see this, let $B=\{0,1\}$, viewed as a lattice with $0<1$. Then, assigning to each set $X$ the set   \[\{\lam:2^X\lra 2\mid \lam(X)=1, \lam(\emptyset)=0, \lam(A\cap B)=\lam(A)\wedge\lam(B)\}\] defines a submonad  of $(\mathbf{P}, \sfm,\sfe)$ --- the proper filter monad; assigning to each set $X$  the set  \[ \{\lam:2^X\lra 2\mid \lam~\text{is a lattice homomorphism}\}\] defines a submonad  of $(\mathbf{P}, \sfm,\sfe)$ --- the    ultrafilter monad. Furthermore, assigning to each set $X$  the set   \[{\sf P}^+(X)\coloneqq\{\lam:2^X\lra 2\mid \lam~\text{is a right adjoint}\}\] defines a submonad  of $(\mathbf{P}, \sfm,\sfe)$, which is essentially the covariant powerset monad.
\end{rem}

\section{The saturated prefilter monad}

For a $\sQ$-semifilter $\mathfrak{F}$ on $X$, the set \begin{equation*}\Gamma_X(\mathfrak{F})=\{\lam\in \sQ^X\mid \mathfrak{F}(\lam)\geq k\}\end{equation*} is   a saturated prefilter on $X$. To see that $\Gamma_X(\mathfrak{F})$ is saturated, suppose that \[\bv_{\lam\in\Gamma_X(\mathfrak{F})} \sub_X(\lam,\mu)\geq k. \] Then, \[\mathfrak{F}(\mu)\geq\bv_{\lam\in\Gamma_X(\mathfrak{F})}\sub_X(\lam,\mu)\with \mathfrak{F}(\lam)= \bv_{\lam\in\Gamma_X(\mathfrak{F})}\sub_X(\lam,\mu)\geq k,\] so $\mu\in \Gamma_X(\mathfrak{F})$.

Conversely, if $\sQ$ is a meet continuous quantale, then for each prefilter $F$ on $X$,    the map \[\Lambda_X(F): \sQ^X\lra \sQ, \quad \Lambda_X(F)(\lam)=\bv_{\mu\in F}\sub_X(\mu,\lam)
\] is   a  $\sQ$-semifilter.

\begin{con}From now on, all quantales are assumed to be meet continuous. \end{con}

The assignment \(X \mapsto \Gamma_X \) is a natural transformation   \[\Gamma:\SFQ\lra\SPF;\]
the assignment \(X \mapsto \Lambda_X \) is a natural transformation  \[\Lambda:\SPF\lra\SFQ.\]
Furthermore, since for each prefilter  $F$   and each $\sQ$-semifilter  $\mathfrak{F}$ on a set $X$,
\begin{align*} F\subseteq\Gamma_X(\frak{F})&\iff \forall\lam\in F,\frak{F}(\lam)\geq k\\
&\iff  \forall\lam\in F,\forall\mu\in \sQ^X,\sub_X(\lam,\mu)\leq\frak{F}(\mu)
\\
&\iff \Lambda_X(F)\leq\frak{F},
\end{align*}   $\Lambda_X$ and $\Gamma_X$ form a Galois connection \cite{Gierz2003} between the partially ordered sets of prefilters and $\sQ$-semifilters on $X$ with $\Lambda_X$ being the left adjoint.

\begin{prop}\label{saturation} For each prefilter $F$ on a set $X$, the saturation of  $F$ is given by \[\Gamma_X\circ\Lambda_X(F)=\Big\{\lam\in\sQ^X\mid \bv_{\mu\in F}\sub_X(\mu,\lam)\geq k\Big\}. \] In particular,  $F$   is  saturated if and only if  $F=\Gamma_X(\mathfrak{F})$ for some $\sQ$-semifilter $\mathfrak{F}$, if and only if $F=\Gamma_X\circ\Lambda_X(F)$.
\end{prop}
\begin{proof} This follows immediately from the adjunction $\Lambda_X\dashv\Gamma_X$ and the fact that \[\lam\in \Gamma_X\circ\Lambda_X(F)\iff \Lambda_X(F)(\lam)=\bv_{\mu\in F}\sub_X(\mu,\lam)\geq k\] for all $\lam\in\sQ^X$. \end{proof}

The adjunction $\Lambda_X\dashv\Gamma_X$ leads to the following:
\begin{defn}A   $\sQ$-semifilter $\mathfrak{F}$ is  conical  if $\frak{F}=\Lambda_X(F)$ for some prefilter $F$. \end{defn}

Let $f:X\lra Y$ be a map and $\mathfrak{F}$  be a conical $\sQ$-semifilter on $X$.
Since \begin{align*}f(\mathfrak{F})(\mu)&=\mathfrak{F}(\mu\circ f) \\ &= \bv_{\lam\in \Gamma_X(\mathfrak{F})}\sub_X(\lam,\mu\circ f) \\ &=\bv_{\lam\in \Gamma_X(\mathfrak{F})}\sub_Y(f(\lam),\mu)\\ &= \bv_{\gamma\in f(\Gamma_X(\mathfrak{F}))} \sub_Y(\gamma,\mu),\end{align*} it follows that $f(\mathfrak{F})$ is   conical. Therefore, we obtain a functor \[\CSFQ:{\sf Set}\lra{\sf Set}.\]

A  $\sQ$-semifilter $\mathfrak{F}$ is  conical if and only if $\Lambda_X\circ\Gamma_X(\frak{F})=\frak{F}$, so, \(F\mapsto\Lambda_X(F)\) establishes a   bijection between saturated prefilters and conical $\sQ$-semifilters on $X$. Since   \[f(\Lambda_X(F))=\Lambda_X(f(F)) \]  for each map $f:X\lra Y$ and each prefilter $F$ on $X$,  it follows that the saturated prefilter functor is naturally isomorphic to the  conical $\sQ$-semifilter functor: \begin{equation} \label{natural iso} \bfig
\morphism(500,0)|b|/{@{>}@<3pt>}/<-500,0>[\CSFQ.`{\sf SPF};\Gamma]
\morphism(0,0)|a|/{@{>}@<3pt>}/<500,0>[{\sf SPF}`\CSFQ.;\Lambda]
\efig\end{equation}

The functor $\CSFQ$ is   a subfunctor of $\SFQ$. Furthermore, it is a retract of $\SFQ$, as we see now. For each set $X$, define \[\mathfrak{c}_X:\SFQ(X)\to\CSFQ(X)\] by letting  \[\mathfrak{c}_X(\mathfrak{F})= \Lambda_X\circ\Gamma_X(\mathfrak{F})\] for each $\sQ$-semifilter $\mathfrak{F}$. Then $\mathfrak{c}_X$ is right adjoint and left inverse to the inclusion \[\mathfrak{i}_X:\CSFQ(X)\lra\SFQ(X).\] This implies that $\mathfrak{c}_X(\mathfrak{F})$ is the largest conical $\sQ$-semifilter that is smaller than or equal to $\mathfrak{F}$, so we call it the \emph{conical coreflection} of $\mathfrak{F}$. We note that for each $\lam\in\sQ^X$, \[\mathfrak{F}(\lam)\geq k\iff \mathfrak{c}_X(\mathfrak{F})(\lam)\geq k.\]

Since $\mathfrak{c}=\{\mathfrak{c}_X\}_X$ and  $\mathfrak{i}=\{\mathfrak{i}_X\}_X$ are   natural transformations, we get natural transformations
  \[{\sf d}\coloneqq\mathfrak{c}\circ\sfe   \quad \text{and}\quad {\sf n} \coloneqq\mathfrak{c}\circ\sfm \circ(\mathfrak{i}*\mathfrak{i}), \]  where    $\sfe$ and $\sfm$  refer to   unit and   multiplication of the monad $(\SFQ,\sfm,\sfe)$, respectively. \[\bfig \qtriangle<550,450>[1_{\sf Set}`\SFQ`\CSFQ;\sfe`\sfd`\mathfrak{c}]\square(1200,0)/>`>`>`<-/<650,450>[\CSFQ^2 `\SFQ^2`\CSFQ`\SFQ; \mathfrak{i}*\mathfrak{i}`\sfn`\sfm`\mathfrak{c}]\efig\]

We spell out the details of   $\sfd$ and $\sfn$ for later use. For each  set $X$ and each $x$ of $X$, since the $\sQ$-semifilter $\sfe_X(x)$ is conical, it follows that\[\sfd_X(x) =\sfe_X(x).\]  For each conical $\sQ$-semifilter  $\mathbb{F}$  on $\CSFQ(X)$ and  each   $\xi:\SFQ(X)\lra\sQ$, since \[(\mathfrak{i}*\mathfrak{i})_X(\mathbb{F})(\xi)= \mathbb{F}(\xi\circ\mathfrak{i}_X),\] it follows that $\sfn_X(\mathbb{F})$,   the conical coreflection of   $\sfm_X\circ(\mathfrak{i}*\mathfrak{i})_X(\mathbb{F}) $, is given by   \[ \mathsf{n}_X(\mathbb{F})(\lam) =\bigvee_{\mathsf{m}_X\circ(\mathfrak{i}*\mathfrak{i})_X(\mathbb{F}) (\mu) \geq k}\sub_X(\mu,\lam) =\bigvee_{\mathbb{F}(\widehat{\mu}\circ \mathfrak{i}_X)\geq k}\sub_X(\mu,\lam) \] for all $\lam\in\sQ^X$.

Let $\Gamma$ and $\Lambda$ be the natural isomorphisms in  (\ref{natural iso}).  Then, \[ \Gamma\circ \sfd\] is a natural transformation   ${\rm id}\lra\SPF$   and \[ \Gamma\circ \sfn\circ (\Lambda*\Lambda)\] is a natural transformation $\SPF^2\lra \SPF$.
The following proposition says that $\Gamma\circ \sfd$ and $\Gamma\circ \sfn\circ (\Lambda*\Lambda)$ are, respectively, the natural transformations $\mathfrak{d}$ and $\mathfrak{n}$  given via the formulas (\ref{defn of d}) and (\ref{defn of n}).

\begin{prop}\label{natural of d and n} $\mathfrak{d} =\Gamma\circ \sfd$  and $\mathfrak{n}=\Gamma\circ \sfn\circ (\Lambda*\Lambda)$.
\end{prop}

\begin{proof} The first equality is obvious. As for the second, let    $\lam\in\sQ^X$. Then   \begin{align*}\lam\in \Gamma_X\circ \sfn_X\circ (\Lambda*\Lambda)_X(\CF)
&\iff \sfm_X\circ(\mathfrak{i}*\mathfrak{i})_X \circ (\Lambda*\Lambda)_X(\CF)(\lam)\geq k\\
&\iff    (\Lambda*\Lambda)_X(\CF)(\widehat{\lam}\circ\mathfrak{i}_X)\geq k \\ &\iff \Lambda_{\SPF(X)}(\CF)(\widehat{\lam}\circ\mathfrak{i}_X\circ\Lambda_X)\geq k\\ &\iff \widehat{\lam}\circ\mathfrak{i}_X\circ\Lambda_X\in\CF.  \quad~~(\CF~\text{is saturated}) \end{align*} Since for each saturated prefilter $F$ on $X$,   \[ \widehat{\lam}\circ\mathfrak{i}_X\circ\Lambda_X(F)=\Lambda_X(F)(\lam)=\bv_{\mu\in F}\sub_X(\mu,\lam) =\widetilde{\lam}(F), \] it follows that $\lam \in \Gamma_X\circ \sfn_X\circ (\Lambda*\Lambda)_X(\CF)$ if and only if $\widetilde{\lam} \in\CF$, hence $\mathfrak{n}=\Gamma\circ \sfn\circ (\Lambda*\Lambda)$. \end{proof}

We say that  conical $\sQ$-semifilters are \emph{closed under multiplication} if for each set $X$ and each conical $\sQ$-semifilter $\mathbb{F}$ on $\CSFQ(X)$, the $\sQ$-semifilter (on $X$) \[\sfm_X\circ(\mathfrak{i}*\mathfrak{i})_X(\mathbb{F})\] is   conical, where $\mathfrak{i}:\CSFQ\lra\SFQ$ denotes the inclusion transformation. By definition, $\sfm_X\circ(\mathfrak{i}*\mathfrak{i})_X(\mathbb{F})$ is  the diagonal $\sQ$-semifilter \[\sfm_X(\mathfrak{i}_X(\mathbb{F})) :\sQ^X\lra\sQ, \quad \lam\mapsto \mathbb{F}(\widehat{\lam}\circ \mathfrak{i}_X).\]

The following conclusion is obvious.

\begin{prop}\label{submonad} If the conical $\sQ$-semifilters are closed under multiplication, then \((\CSFQ, \sfn,\sfd)\)   is a  submonad of $(\SFQ,\sfm,\sfe)$. \end{prop}

There is an easy-to-check condition for $\CSFQ$ to be a submonad of the $\sQ$-semifilter monad:

\begin{prop}\label{easy one} Let $\sQ$ be a quantale such that for each $p\in \sQ$, the map \[p\ra-: \sQ\lra \sQ\] preserves directed joins. Then,   $\CSFQ$ is a  submonad  of the $\sQ$-semifilter monad. \end{prop}

We prove two lemmas first. The first  is a slight extension of   \cite[Proposition 9]{Garcia2006}; the second is  a slight extension of \cite[Corollary 3.13]{LZ2018}). Proofs are included here for convenience of the reader.

\begin{lem}\label{coni filt1}   A $\sQ$-semifilter $\mathfrak{F}$ on a  set $X$   is conical if and only if  $$\frak{F}(\lam)=\bv\{p\in \sQ\mid \frak{F}(p\ra\lam)\geq k\}$$ for all $\lam\in \sQ^X$. \end{lem}

\begin{proof} \textbf{Necessity}.  Suppose that $\mathfrak{F}$ is   conical. If  $\frak{F}(p\ra\lam)\geq k$, then \[p\leq \sub_X(p\ra\lam,\lam)\leq \mathfrak{F}(p\ra\lam)\ra\mathfrak{F}(\lam)\leq\mathfrak{F}(\lam),\]  hence   $$\frak{F}(\lam)\geq\bv\{p\in\sQ\mid \frak{F}(p\ra\lam)\geq k\}.$$ As for the converse inequality,  for each $\nu\in\Gamma_X(\mathfrak{F})$, let $p_\nu=\sub_X(\nu,\lam)$. Since $\nu\leq p_\nu\ra \lam$, then $\mathfrak{F}(p_\nu\ra\lam)\geq k$, and consequently,    $$\frak{F}(\lam)=\bv_{\nu\in \Gamma_X(\mathfrak{F})}\sub_X(\nu,\lam)\leq\bv\{p\in\sQ\mid \frak{F}(p\ra\lam)\geq k\}.$$

\textbf{Sufficiency}. It suffices to check that $\frak{F}\leq\Lambda_X\circ\Gamma_X(\frak{F})$ since $\Lambda_X$ is left adjoint to $\Gamma_X$. For all $p\in\sQ$ with $\frak{F}(p\ra\lam)\geq k$,  one has $p\ra\lam\in\Gamma_X(\frak{F})$,  so, $$p\leq \sub_X(p\ra\lam,\lam)\leq\Lambda_X\circ\Gamma_X(\frak{F})(\lam).$$ Therefore,  $\frak{F}(\lam)\leq\Lambda_X\circ\Gamma_X(\frak{F})(\lam)$ since $\frak{F}(\lam)=\bv\{p\in\sQ\mid \frak{F}(p\ra\lam)\geq k\}$. \end{proof}

\begin{lem}\label{confilt=strong} Let $\sQ$ be a quantale such that for each $p\in \sQ$, the map $p\ra-: \sQ\lra \sQ$ preserves directed joins. Then,  a $\sQ$-semifilter $\mathfrak{F}$  on a set $X$  is conical if and only if \[\mathfrak{F}(p\ra\lam)=p\ra \mathfrak{F}(\lam)\] for all $p\in\sQ$ and $\lam\in\sQ^X$. \end{lem}

\begin{proof} If $\mathfrak{F}$ is conical, then for each $p\in\sQ$, \begin{align*} \mathfrak{F} (p\ra\lam)&=\bv_{\nu\in\Gamma_X(\frak{F})}\sub_X(\nu,p\ra\lam)\\ &=\bv_{\nu\in\Gamma_X(\frak{F})}p\ra \sub_X(\nu,\lam)\\ &=p\ra\bv_{\nu\in\Gamma_X(\frak{F})}\sub_X(\nu,\lam)\\ &=p\ra\frak{F}(\lam).\end{align*}  Conversely,  assume that $\mathfrak{F}(p\ra\lam)=p\ra \mathfrak{F}(\lam)$ for all $p\in\sQ$ and $\lam\in\sQ^X$. Then  \[p\leq \mathfrak{F}(\lam)\iff \mathfrak{F}(p\ra\lam)\geq k,\] hence \[\frak{F}(\lam)=\bv\{p\in \sQ\mid \frak{F}(p\ra\lam)\geq k\},\] and consequently, $\mathfrak{F}$ is conical by Lemma \ref{coni filt1}.  \end{proof}

\begin{proof}[Proof of Proposition \ref{easy one}] It suffices to check that conical $\sQ$-semifilters are closed under multiplication. Let $\mathbb{F}$ be a conical $\sQ$-semifilter  on   $\CSFQ(X)$. We need to check that the diagonal $\sQ$-semifilter $\sfm_X(\mathfrak{i}_X(\mathbb{F}))$ is conical. For each $\lam\in\sQ^X$ and $p\in \sQ$, since \begin{align*}\sfm_X(\mathfrak{i}_X(\mathbb{F}))(p\ra\lam)&= \mathfrak{i}_X(\mathbb{F})(\widehat{p\ra\lam}) \\ & =\mathbb{F}(p\ra(\widehat{\lam}\circ \mathfrak{i}_X))\\ & = p\ra \mathbb{F}(\widehat{\lam}\circ \mathfrak{i}_X)\\ &=p\ra \sfm_X(\mathfrak{i}_X(\mathbb{F}))(\lam),\end{align*}  then the conclusion follows immediately from  Lemma \ref{confilt=strong}. \end{proof}

\begin{cor}\label{easyone-b} Let $\sQ$ be a quantale such that for each $p\in \sQ$,  the map $p\ra-: \sQ\lra \sQ$ preserves directed joins. Then,   $(\SPF, \mathfrak{n},\mathfrak{d})$ is a monad  on the category of  sets. \end{cor}

The condition in Proposition \ref{easy one}   is  a  strong one. For example,  for a continuous t-norm $\&$,  the quantale $\sQ=([0,1],\with,1)$ satisfies the condition in Proposition \ref{easy one} (i.e., $p\ra-:[0,1]\lra[0,1]$ preserves directed joins for all $p$)  if and only if  $\&$ is Archimedean.  So, there are essentially only two such  t-norms --- the product t-norm and the \L ukasiewicz t-norm.

It seems hard to find a sufficient and necessary condition for a general quantale $\sQ$ so that \((\CSFQ, \sfn,\sfd)\) is monad. However, in the case that $\sQ$ is the interval $[0,1]$ equipped with a continuous t-norm, there is one.

\begin{prop}{\rm (\cite{LZ2020})} For each continuous t-norm $\&$    on $[0,1]$, the following conditions are equivalent: \begin{enumerate}[label=\rm(S\arabic*)]\item For each non-idempotent element $a\in[0,1]$, the quantale $([a^-,a^+],\with ,a^+)$ is isomorphic to $([0,1],\with_P,1)$   whenever $a^->0$.
\item The implication   $\ra:[0,1]^2\to[0,1]$ is  continuous  at every point off the diagonal $\{(x,x)\mid x\in[0,1]\}$.
\item For each $p\in(0,1]$, the map $p\ra-:[0,p)\lra[0,1]$   preserves directed joins.  \end{enumerate} \end{prop}

We say that a continuous t-norm satisfies the \emph{condition (S)} if it satisfies one, hence all, of the equivalent conditions (S1)--(S3). Every continuous Archimedean t-norm satisfies the condition (S); the ordinal sum decomposition theorem guarantees that there exist many continuous t-norms that satisfy the condition (S)  but are not Archimedean, with the G\"{o}del t-norm being an example.

\begin{thm}\label{main1} Let $\sQ=([0,1],\&,1)$ with $\&$ being a continuous t-norm.
Then  the following statements are equivalent: \begin{enumerate}[label=\rm(\arabic*)]  \item The t-norm $\&$ satisfies the condition (S).
 \item   $\CSFQ$ is a submonad of $(\SFQ,\sfm,\sfe)$.
 \item  The triple $(\CSFQ, \sfn,\sfd)$ is a monad. \item The triple $(\SPF, \mathfrak{n},\mathfrak{d})$ is a monad.
         \end{enumerate}
     \end{thm}

We make some preparations first.

\begin{lem}\label{CSF is closed under directed joins}
If $\{\mathfrak{F}_i\}_i$ is a directed set of conical $\sQ$-semifilters on $X$, then so is $\bv_i\mathfrak{F}_i$. \end{lem}

\begin{proof} Let $F=\bigcup_i\Gamma_X(\mathfrak{F}_i)$. Then  $F$ is a prefilter and $\Lambda_X(F)=\bv_i\mathfrak{F}_i$. \end{proof}

\begin{prop}\label{criterion} Conical $\sQ$-semifilters  are closed under multiplication if and only if the following conditions are satisfied: \begin{enumerate}[label=\rm(\arabic*)] \item If $\{\mathfrak{F}_i\}_i$ is a set of conical $\sQ$-semifilters on $X$, then so is $\bw_i\mathfrak{F}_i$.
\item If $\mathfrak{F} $ is a conical $\sQ$-semifilter   on $X$, then so is $p\ra\mathfrak{F} $ for each $p\in\sQ$. \end{enumerate}
\end{prop}

\begin{proof} \textbf{Sufficiency}. If we can show that for each conical $\sQ$-semifilter $\mathbb{F}$    on the set $\CSFQ(X)$,  \[\sfm_X\circ(\mathfrak{i}*\mathfrak{i})_X(\mathbb{F})= \bv_{\mathbb{F}(\xi)\geq k} \bw_{\mathfrak{G}\in \CSFQ(X)}(\xi(\mathfrak{G})\ra \mathfrak{G}),\] then the conclusion follows immediately since each directed join of conical $\sQ$-semifilters is conical.

 In fact, for each $\lam\in \sQ^X$, \begin{align*} \sfm_X\circ(\mathfrak{i}*\mathfrak{i})_X (\mathbb{F})(\lam)&= \mathbb{F}(\widehat{\lam}\circ \mathfrak{i}_X) \\ &= \bv_{\mathbb{F}(\xi)\geq k}\sub_{\CSFQ(X)}(\xi,\widehat{\lam}\circ \mathfrak{i}_X) \quad~~ (\mathbb{F}~\text{is conical})\\
 &= \bv_{\mathbb{F}(\xi)\geq k} \bw_{\mathfrak{G}\in \CSFQ(X)}(\xi(\mathfrak{G})\ra \mathfrak{G})(\lam).\end{align*}

\textbf{Necessity}. Let  $F$ be the prefilter on $\CSFQ(X)$ consisting of maps $\xi:\CSFQ(X)\lra\sQ$ with $\xi(\mathfrak{F}_i)= 1$ for all $i$  and let $\mathbb{F}=\Lambda_X(F)$.   Then, for each $\lam\in\sQ^X$, \begin{align*}\sfm_X\circ(\mathfrak{i}*\mathfrak{i})_X (\mathbb{F})(\lam)= \mathbb{F}(\widehat{\lam}\circ \mathfrak{i}_X) =\bv_{\xi\in F}\sub_{\CSFQ(X)}(\xi,\widehat{\lam}\circ \mathfrak{i}_X)= \bw_i\widehat{\lam}(\mathfrak{F}_i) =\bw_i\mathfrak{F}_i(\lam), \end{align*} which shows that $\bw_i\mathfrak{F}_i=\sfm_X\circ(\mathfrak{i}*\mathfrak{i})_X(\mathbb{F})$, so $\bw_i\mathfrak{F}_i$ is conical. This proves (1).

As for (2),  assume that $\mathfrak{F}$ is a conical $\sQ$-semifilter on $X$ and $p\in\sQ$. Let $F$ be the prefilter on $\CSFQ(X)$ consisting of maps $\xi:\CSFQ(X)\lra\sQ$ with $\xi(\mathfrak{F})\geq p$  and let $\mathbb{F}=\Lambda_X(F)$. Then, for each $\lam\in\sQ^X$, \[\sfm_X\circ(\mathfrak{i}*\mathfrak{i})_X(\mathbb{F})(\lam)= \mathbb{F}(\widehat{\lam}\circ\mathfrak{i}_X)=\bv_{\xi\in F}\sub_{\CSFQ(X)}(\xi,\widehat{\lam}\circ\mathfrak{i}_X)=p\ra \mathfrak{F}(\lam).\] Therefore, $p\ra \mathfrak{F}=\sfm_X\circ(\mathfrak{i}*\mathfrak{i})_X(\mathbb{F})$ and is, consequently, conical. \end{proof}

\begin{lem}\label{coni filt2}Let $\sQ$ be a continuous quantale. Then, a $\sQ$-semifilter $\mathfrak{F}$    is conical if and only if   $\frak{F}(p\ra\lam)\geq k$ whenever $p\ll\frak{F}(\lam)$. \end{lem}

\begin{proof}    It suffices to check that $\frak{F}\leq\Lambda\circ\Gamma(\frak{F})$. If $p\ll\frak{F}(\lam)$,  then $\frak{F}(p\ra\lam)\geq k$ by assumption. Hence, $p\ra\lam\in\Gamma(\frak{F})$, so, $p\leq \sub_X(p\ra\lam,\lam)\leq\Lambda\circ\Gamma(\frak{F})(\lam)$, and consequently, $\frak{F}(\lam)\leq\Lambda\circ\Gamma(\frak{F})(\lam)$. \end{proof}

\begin{prop}\label{CSF is closed under meets} Let $\sQ$ be a continuous quantale. If each member of  $\{\mathfrak{F}_i\}_i$ is a   conical $\sQ$-semifilter  on a set $X$, then so is the meet $\bw_i\mathfrak{F}_i$.  \end{prop}
\begin{proof} If $p\ll\bw_i\mathfrak{F}_i(\lam)$, then for each $i$, $p\ll\mathfrak{F}_i(\lam)$, hence $\mathfrak{F}_i(p\ra\lam)\geq k$, and consequently, $\bw_i\mathfrak{F}_i(p\ra\lam)\geq k$. Therefore, the conclusion holds by Lemma \ref{coni filt2}. \end{proof}

\begin{proof}[Proof of Theorem \ref{main1}] That $(3)\Leftrightarrow(4)$ is clear. We prove that $(1)\Rightarrow(2)\Rightarrow(3)\Rightarrow(1)$.

$(1)\Rightarrow(2)$ Since $([0,1],\&,1)$ is a continuous quantale, by Propositions  \ref{criterion} and   \ref{CSF is closed under meets}, it suffices to prove that for each conical $\sQ$-semifilter $\mathfrak{F}:[0,1]^X\lra [0,1]$ and each $p\in[0,1]$,   $p\ra \mathfrak{F}$  is   conical. Let \[G=\{\lam\in [0,1]^X\mid p\leq \mathfrak{F}(\lam)\}.\] Then $G$ is    a prefilter on $X$. We show  in two steps that \[p\ra\mathfrak{F} =\Lambda_X(G),\] which implies that $p\ra\mathfrak{F}$ is  conical.

\textbf{Step 1}. $\Lambda_X(G)\leq p\ra\mathfrak{F}$.

For each $\lam\in G$, since $\mathfrak{F}(\lam)\geq p$, it follows that for all $\mu\in[0,1]^X$, \[\sub_X(\lam,\mu)\leq \mathfrak{F}(\lam)\ra \mathfrak{F}(\mu)\leq p\ra \mathfrak{F}(\mu),\]  hence $\sub_X(\lam,-)\leq p\ra\mathfrak{F}$ and consequently, \[\Lambda_X(G)=\bv_{\lam\in G}\sub_X(\lam,-)\leq p\ra\mathfrak{F}.\]

\textbf{Step 2}.   $p\ra\mathfrak{F} \leq\Lambda_X(G)$.

First, we claim that if $\lam \in \Ga_X(\mathfrak{F})$  then  $p\with\lam \in G$. In fact, if $\mathfrak{F}(\lam)=1$, then \[p\leq\sub_X(\lam,p\with\lam)\leq \mathfrak{F}(\lam)\ra\mathfrak{F}(p\with\lam)= \mathfrak{F} (p\with\lam),\] hence  $p\with\lam \in G$.

Now we check that $p\ra\mathfrak{F}(\lam)\leq\Lambda_X(G)(\lam)$ for all $\lam\in[0,1]^X$.

If $p\leq\mathfrak{F}(\lam)$, then  $\lam\in G$ and consequently, \[p\ra\mathfrak{F}(\lam)\leq1=\sub_X(\lam, \lam)\leq \bv_{\mu\in G} \sub_X(\mu,\lam)= \Lambda_X(G)(\lam).\]
If  $p>\mathfrak{F}(\lam)$, then \begin{align*}p\ra\mathfrak{F}(\lam)& =p\ra\bv_{\mu\in\Gamma_X(\mathfrak{F})}\sub_X(\mu,\lam) &\text{($\mathfrak{F}$ is conical)}\\ &
 =\bv_{\mu\in\Gamma_X(\mathfrak{F})}(p\ra\sub_X(\mu,\lam)) &\text{(condition (S))}
 \\ &
 =\bv_{\mu\in\Gamma_X(\mathfrak{F})} \sub_X(p\with\mu,\lam)  \\ &
 \leq\bv_{\gamma\in G} \sub_X(\gamma,\lam)\\ &
 =\Lambda_X(G)(\lam).
\end{align*}

$(2)\Rightarrow(3)$ Clear by definition.

$(3)\Rightarrow(1)$
For each map $h:X\lra\CSFQ(Y)$, let $h^\sharp $ be the composite \[  \CSFQ(X)\to^{\CSFQ(h)}\CSFQ^2(Y)\to^{\sfn_Y}\CSFQ(Y). \] Since $(\CSFQ,\sfn,\sfd)$ is a monad, it follows that \[  g^\sharp \circ f^\sharp = (g^\sharp \circ f)^\sharp   \] for any $f:X\lra\CSFQ(Y)$ and $g:Y\lra\CSFQ(Z)$, see e.g. \cite{Manes1976,Manes2003}.

In the following we derive a contradiction if $ \& $ does not satisfy the condition (S). Suppose that $p,q\in[0,1]$ are idempotent elements of $\&$ such that $0<p<q$  and that the restriction of $ \& $ on $[p,q]$ is isomorphic to the {\L}ukasiewicz t-norm.
Pick  $t,s\in(p,q)$ with $t\with s=p$.

Let $X=[0,1]$.
Consider the constant maps \[f:X\lra\CSFQ(X) \quad\text{and}\quad g:X\lra\CSFQ(X)  \] given as follows: $f$ sends every $x$ to   the conical $\sQ$-semifilter  $\mathfrak{F}$ generated by the prefilter \[\{\nu:X\lra[0,1]\mid \nu \geq t_X   \};\]  $g$ sends every $x$ to the conical $\sQ$-semifilter  $\mathfrak{G}$ on $X$ generated by the prefilter \[\{\nu:X\lra[0,1]\mid \exists n\geq 1, \nu\geq 1_{A_n}\}, \] where $A_n=\{ 1/m \mid m\geq n\}$. By definition, for all $\mu\in[0,1]^X$, \[f(x)(\mu)=\mathfrak{F}(\mu)= t\ra\bw_{y\in X}\mu(y)   \]   and   $$ g(x)(\mu)=\mathfrak{G}(\mu)  = \bigvee_{n\geq 1}\bw_{m\geq n}\mu(1/m).$$

We claim that $g^\sharp \circ f^\sharp \not= (g^\sharp \circ f)^\sharp$, contradicting that $(\CSFQ,\sfn,\sfd)$ is a monad. To see this, let $\gamma(x)=p(1-x)$; let $\mathfrak{H}$ be the conical $\sQ$-semifilter    generated by the prefilter \[\{\nu:X\lra[0,1]\mid \nu \geq s_X  \},\] that is to say, \[\mathfrak{H}(\mu)=s\ra \bw_{y\in X}\mu(y).\] In the following we show in two steps that $g^\sharp \circ f^\sharp (\mathfrak{H})(\gamma)$ is not equal to $(g^\sharp \circ f)^\sharp(\mathfrak{H})(\gamma)$.

\textbf{Step 1}.   $g^\sharp \circ f^\sharp (\mathfrak{H})(\gamma)=1$.

Since \begin{align*} \sfm_X(f(\mathfrak{H}))(\lam)
&= f(\mathfrak{H})(\widehat{\lam}) \\ &= \mathfrak{H}(\widehat{\lam}\circ f) \\ &=  \mathfrak{H}(x\mapsto \mathfrak{F}(\lam))\\
&=s\ra \mathfrak{F}(\lam),
\end{align*} it follows that $f^\sharp (\mathfrak{H})$ is the conical coreflection of  $s\ra \mathfrak{F}$. Since \[s\ra \mathfrak{F}(\mu)= s\ra\Big(t\ra \bw_{y\in X}\mu(y)\Big)=p\ra\bw_{y\in X}\mu(y),\]  it follows that $s\ra \mathfrak{F}$ is generated by $\{\nu\mid \nu \geq p_X\}$, so $s\ra \mathfrak{F}$ is conical and   \[f^\sharp (\mathfrak{H})(\mu)=p\ra \bw_{y\in X}\mu(y).\]

Since
\begin{align*} \sfm_X(g(f^\sharp (\mathfrak{H})))(\lam)
&= g(f^\sharp (\mathfrak{H}))(\widehat{\lam}) \\ &= f^\sharp (\mathfrak{H})(\widehat{\lam}\circ g) \\ &=  f^\sharp (\mathfrak{H})(x\mapsto \mathfrak{G}(\lam))\\
&=p\ra \mathfrak{G}(\lam),
\end{align*}
it follows that $g^\sharp \circ f^\sharp (\mathfrak{H}) $ is the conical coreflection of $p\ra \mathfrak{G}$.

Because \[p\ra \mathfrak{G}(\mu)=1 \iff p\leq \mathfrak{G}(\mu) = \bv_n\bw_{m\geq n}\mu(1/m), \]therefore \begin{align*}g^\sharp \circ f^\sharp (\mathfrak{H})(\gamma) &=\bv\Big\{\sub_X(\mu,\gamma)\mid p\leq \bv_n\bw_{m\geq n}\mu(1/m) \Big\}  =1. \end{align*}

\textbf{Step 2}. $ (g^\sharp \circ f)^\sharp(\mathfrak{H})(\gamma)\leq p$.

Since \begin{align*}\sfm_X(g^\sharp \circ f(\mathfrak{H})) (\mu) &=g^\sharp \circ f(\mathfrak{H}) (\widehat{\mu})\\ & = \mathfrak{H} (\widehat{\mu}\circ g^\sharp \circ f)\\ &= \mathfrak{H}(x\mapsto g^\sharp (\mathfrak{F})(\mu))\\ & = s\ra g^\sharp (\mathfrak{F})(\mu),\end{align*} it follows that
$ (g^\sharp \circ f)^\sharp (\mathfrak{H})$ is   the conical coreflection of $s\ra g^\sharp (\mathfrak{F})$.

Since  \begin{align*} \sfm_X(g(\mathfrak{F}))(\lam)
&= g(\mathfrak{F})(\widehat{\lam}) \\ &= \mathfrak{F}(\widehat{\lam}\circ g) \\ &=  \mathfrak{F}(x\mapsto \mathfrak{G}(\lam))\\
&= t\ra \mathfrak{G}(\lam),
\end{align*}
it follows that $g^\sharp (\mathfrak{F}) $ is the conical coreflection of $t\ra \mathfrak{G}$. Thus,
\begin{align*} g^\sharp (\mathfrak{F})(\mu) &= \bv\Big\{\sub_X(\lam,\mu)\mid  t \leq \mathfrak{G}(\lam) \Big\}\\ &= \bv\Big\{\sub_X(\lam,\mu)\mid  t \leq \bv_n\bw_{m\geq n}\lam(1/m) \Big\}. \end{align*}

If $s\leq g^\sharp (\mathfrak{F})(\mu)$, then there exist $r>p$ and $\lam\in [0,1]^X$ such that $t \leq \bv_n\bw_{m\geq n}\lam(1/m)$ and that $r<\sub_X(\lam,\mu)$. Since $t,r>p$, then for $m$ large enough, $\mu(1/m)\geq p$ and consequently, \[\mu(1/m)\ra\gamma(1/m)=\mu(1/m)\ra (p(1-1/m))=p(1-1/m)\leq p.\]   Therefore, \begin{align*}(g^\sharp \circ f)^\sharp(\mathfrak{H})(\gamma)&=\bv\big\{\sub_X(\mu,\gamma)\mid s\leq g^\sharp (\mathfrak{F})(\mu)\big\}\leq p. \qedhere \end{align*} \end{proof}

\section{The $\top$-filter monad}

This section considers the question whether   $\top$-filters  give rise to a monad. The idea is to relate $\top$-filters to conical $\sQ$-filters. By a \emph{conical $\sQ$-filter} we mean, of course, a $\sQ$-semifilter that is conical and is a $\sQ$-filter.

Assigning to each set $X$  the set $\CFQ(X)$ of all conical $\sQ$-filters on $X$ defines a subfunctor of $\FQ$: \[\CFQ:{\sf Set}\lra{\sf Set}.\]

For a meet continuous and integral  quantale $\sQ$, the functor $\topF$  is naturally isomorphic to   $\CFQ$, as we see now.

\begin{prop}Let $\sQ$ be a meet continuous and integral  quantale. Then a saturated prefilter $F$ on a  set $X$ is a $\top$-filter if and only if $\Lambda_X(F)$ is a $\sQ$-filter. \end{prop}

\begin{proof}If $F$ is a $\top$-filter then for each $p\in\sQ$, \[\Lambda_X(F)(p_X) = \bv_{\mu\in F}\bw_{x\in X}(\mu(x)\ra p) = \bv_{\mu\in F}\Big(\Big(\bv_{x\in X}\mu(x)\Big)\ra p\Big)=p, \] showing that $\Lambda_X(F)$ is   a $\sQ$-filter. If $F$ is not a $\top$-filter, there is some $\mu\in F$ such that $\bv_{x\in X}\mu(x)=q<1$. Then, \[\Lambda_X(F)(q_X) \geq \bw_{x\in X}(\mu(x)\ra q)=1>q,\] which implies that $\Lambda_X(F)$ is not a $\sQ$-filter. \end{proof}
Therefore, the correspondence  $F\mapsto\Lambda_X(F)$   establishes a   bijection between $\top$-filters and conical $\sQ$-filters and consequently,  $\topF$  is naturally isomorphic to  $\CFQ$: $$\bfig
\morphism(550,0)|b|/{@{>}@<3pt>}/<-550,0>[\CFQ.`\topF;\Gamma]
\morphism(0,0)|a|/{@{>}@<3pt>}/<550,0>[\topF`\CFQ.;\Lambda]
\efig$$

\begin{prop}Let $\sQ$ be a meet continuous and integral  quantale. Then the conical coreflection of each $\sQ$-filter is a $\sQ$-filter. \end{prop}

\begin{proof}It suffices to check that for each $\sQ$-filter  $\mathfrak{F}$ on a set $X$,  $\Gamma_X(\mathfrak{F})$ is a $\top$-filter. If $\Gamma_X(\mathfrak{F})$ is not a $\top$-filter, then, by the argument of the above proposition,  there is some $q<1$ such that $\Lambda_X(\Gamma_X(\mathfrak{F}))(q_X)=1$, contradicting that $\Lambda_X(\Gamma_X(\mathfrak{F}))(q_X)\leq \mathfrak{F}(q_X)=q$. \end{proof}

Therefore, the natural transformation $\mathfrak{c}:\SFQ\lra\CSFQ$  restricts to a  natural transformation $\mathfrak{c}:\FQ\lra\CFQ$. In particular,  $\CFQ$ is a retract of $\FQ$.

Since $\FQ$ is   a submonad of $(\SFQ,\sfm,\sfe)$, it follows
 that
  \begin{equation}\label{def of n2} {\sf n} \coloneqq \mathfrak{c}\circ\sfm \circ(\mathfrak{i}*\mathfrak{i})  \end{equation}  is a natural transformation $\CFQ^2\lra\CFQ$.\[\bfig \square(1200,0)/>`>`>`<-/<650,450>[\CFQ^2 `\SFQ^2`\CFQ`\SFQ; \mathfrak{i}*\mathfrak{i}`\sfn`\sfm`\mathfrak{c}]\efig\]
Consequently,  \begin{equation}\label{def of fn2}\mathfrak{n}\coloneqq\Gamma\circ \sfn\circ (\Lambda*\Lambda)\end{equation} is  a   natural transformation  $\topF^2\lra\topF$.

It is trivial that $\sfd_X(x)$ is a conical $\sQ$-filter for all set $X$ and all $x\in X$,
so, \[\sfd=\{\sfd_X\}_X\] is a natural transformation ${\rm id}\lra\CFQ$ and  \[\mathfrak{d}=\{\mathfrak{d}_X\}_X\] is  a natural transformation  ${\rm id}\lra\topF$. Now, we   state the question of this section.

\begin{ques}\label{problem2} Let $\sQ$ be a meet continuous and integral  quantale;   let $\sfn:\CFQ^2\lra\CFQ$ and $\mathfrak{n}:\topF^2\lra\topF$ be the natural transformations given in  (\ref{def of n2}) and (\ref{def of fn2}), respectively.
When is the triple \[(\topF, \mathfrak{n},\mathfrak{d})\]  a monad? Or equivalently,  when is the triple \[(\CFQ, \sfn,\sfd)\]  a  monad?
\end{ques}

\begin{rem}\label{6.4}It follows from Proposition \ref{natural of d and n} that the natural transformations $\mathfrak{d}:{\rm id}\lra\topF$ and $\mathfrak{n}:\topF^2\lra\topF$ coincide, respectively, with the natural transformations $\eta$ and $\mu$  in    \cite[Section 3]{YF2020}    defined  via    (\ref{defn of d}) and (\ref{defn of n}).
\end{rem}

Similar to Proposition \ref{easy one} and Corollary \ref{easyone-b}, it can be shown that if $\sQ$ is a meet continuous and  integral  quantale such that  the map $p\ra-: \sQ\lra \sQ$ preserves directed joins for each $p\in \sQ$, then,   $\CFQ$ is a  submonad  of the $\sQ$-semifilter monad, hence  $(\topF, \mathfrak{n},\mathfrak{d})$  is a monad  \cite[Lemma 3.1]{YF2020}.

The main result of this section presents an answer to Question \ref{problem2} in the case that   $\sQ$ is the interval $[0,1]$ equipped with a continuous t-norm.

\begin{thm}\label{main2} Let $\sQ=([0,1],\&,1)$ with $\&$ being a continuous t-norm.
Then  the following statements are equivalent: \begin{enumerate}[label=\rm(\arabic*)]  \item The t-norm $\&$ satisfies the condition (S).
\item  $\CFQ$ is a submonad of $(\SFQ,\sfm,\sfe)$.
\item  The triple $(\CFQ, \sfn,\sfd)$ is a monad.
\item  The triple $(\topF, \mathfrak{n},\mathfrak{d})$ is a monad.
\end{enumerate}
     \end{thm}

\begin{proof}
$(1)\Rightarrow(2)$ By Theorem \ref{main1}, we know that  conical $\sQ$-semifilters are closed under multiplication in this case. Since $\sQ$-filters are   closed under multiplication, it follows that  conical $\sQ$-filters are   closed under multiplication. Therefore,  $\CFQ$ is a submonad of   $(\SFQ,\sfm,\sfe)$.

$(2)\Rightarrow(3)$ By definition of $\sfn$ and $\sfd$.

$(3)\Rightarrow(1)$ A slight improvement of the proof of $(3)\Rightarrow(1)$  in Theorem \ref{main1} will suffice. Let  $\mathfrak{F}$ be  the conical $\sQ$-filter   generated by the prefilter \[\{\nu:X\lra[0,1]\mid \nu(1)=1,  \nu \geq t_X \};\] let  $\mathfrak{G}$ be the conical $\sQ$-filter  on $X$ generated by the prefilter \[\{\nu:X\lra[0,1] \mid \nu\geq 1_{A_n} ~\text{for  some}~n\geq 1  \} ,\] where $A_n=\{ 1/m \mid m\geq n\}$.

Consider the maps \[f,g:X\lra\CFQ(X)\] given by \[f(x)=\begin{cases} \mathfrak{F},& x<1,\\ \sfd_X(1), &x=1\end{cases} \] and  \[g(x)=\begin{cases} \mathfrak{G},& x<1,\\ \sfd_X(1), &x=1.\end{cases} \] Then, via similar calculations, one sees that $g^\sharp \circ f^\sharp (\mathfrak{H})(\gamma)$ is not equal to $ (g^\sharp \circ f)^\sharp(\mathfrak{H})(\gamma)$, where  \[\gamma(x)=\begin{cases}p(1-x), & x<1, \\ 1, & x=1\end{cases}\] and   $\mathfrak{H}$ is the  $\sQ$-filter    generated by the prefilter \(\{\nu:X\lra[0,1]\mid \nu(1)=1,\nu \geq s_X  \}.\)

$(3)\Leftrightarrow(4)$ By definition.
\end{proof}

\section{The   bounded saturated prefilter monad}
Functional ideals  \cite{Lowen15,LOV2008,LV2008},    postulated for Lawvere's quantale $([0,\infty]^{\rm op},+,0)$, are a special kind of saturated prefilters. The notion of bounded  saturated prefilters  is an extension of that of functional ideals to the quantale-valued context with the quantale being the interval $[0,1]$ equipped with a  continuous t-norm $\&$.  This section investigates whether   bounded  saturated prefilters give rise to a monad.\footnote{We thank gratefully Dirk Hofmann  for bringing this question to our attention.}

A map $ \lambda: X\lra[0,1] $ is said to be \emph{bounded} (precisely, bounded below), if    $\lambda\geq \epsilon_X $ for some $\epsilon>0$.

Let $\with$ be a continuous t-norm on $[0,1]$; let $F$ be a  saturated prefilter on a   set $ X $. Since $[0,1]$ is linearly ordered,   \[B_F\coloneqq\{\mu\in F\mid\text{$ \mu $ is bounded}\} \] is   a prefilter on $X$.  We say that $F$ is a \emph{bounded saturated prefilter}  if  it is the saturation of $B_F$; that is to say,   \[\lam\in F\iff \bv_{\mu\in B_F}\sub_X(\mu,\lam)=1.\]

Let $ \BSF(X) $ denote the set of bounded saturated prefilters on $ X $.

\begin{prop}\label{sat is bounded} For a continuous t-norm $\&$, the following statements are equivalent: \begin{enumerate}[label=\rm(\arabic*)] \item Every saturated prefilter is bounded. \item  $\& $ is isomorphic to the {\L}ukasiewicz t-norm.   \end{enumerate}
\end{prop}

\begin{lem}\label{LT} A continuous t-norm
$ \& $ is isomorphic to the {\L}ukasiewicz t-norm if and only if \[ \bigvee_{p>0}(p\ra 0)=1 .\]
\end{lem}

\begin{proof}Necessity is clear. As for  sufficiency, first  we show that $\&$ has no idempotent element in $ (0,1)$, hence $\&$ is either isomorphic to the {\L}ukasiewicz t-norm or to the product t-norm.  If, on the contrary, $b\in(0,1)$ is an idempotent element, then for each $p>0$, $p\ra 0\leq b<1$, a contradiction. If $\&$ is  isomorphic to the product t-norm, then for each $p>0$, $p\ra 0=0$, a contradiction.  Therefore, $\&$ is  isomorphic to the {\L}ukasiewicz t-norm. \end{proof}

\begin{proof}[Proof of Proposition \ref{sat is bounded}]
$(1)\Rightarrow(2)$ Consider the largest prefilter $F$ on a singleton set; that is, $F$ is   the unit interval $[0,1]$. By assumption, $F$ is bounded. Then   \[  \bigvee_{p>0}(p\ra 0)=1, \] which, by Lemma \ref{LT}, implies that $\&$ is isomorphic to the {\L}ukasiewicz t-norm.

$(2)\Rightarrow(1)$ Without loss of generality, we may assume that $\&$ is, not only  isomorphic to, the {\L}ukasiewicz t-norm. Let $F$ be a saturated prefilter on a set $X$. Since for all $p>0$ and $\mu\in F$, $p_X\vee\mu$ is bounded and \[\sub_X(p_X\vee\mu,\mu)= \bw_{x\in X}(p \ra\mu(x))\geq 1-p,\] it follows that \[ 1=\bigvee_{p>0}\sub_X(p_X\vee\mu,\mu). \] Therefore, $F$ is bounded.
\end{proof}

\begin{con}Because of Proposition \ref{sat is bounded},  in the remainder of this section, we always assume that $\sQ$ is the quantale $([0,1],\&,1)$ with  $\&$ being a continuous t-norm that is not isomorphic to  the {\L}ukasiewicz t-norm, unless otherwise specified. \end{con}

\begin{prop}\label{BS}
Every element of a bounded saturated prefilter  is bounded.
\end{prop}
\begin{proof}
Let $F$ be a bounded saturated prefilter and $B_F$ be the set of all bounded elements of $F$. By definition, for each $\mu\in F$, we have 	  \[ \bigvee_{\lam\in B_F}\sub_X(\lambda,\mu)=1. \]
Let $b$ be an idempotent element of $\&$ in $ (0,1)$ if  $\&$ has one, otherwise, take an arbitrary element in $(0,1)$ for $b$.  Then there exist some $\lam\in F$ and $a>0$ such that $\lam\geq a_X$ and $b\leq\sub_X(\lam,\mu)$. Then for each $x$, we have \(0<a\with b\leq \mu(x)\), showing that $\mu$ is bounded.
\end{proof}

\begin{exmp}[Functional ideals, I] \label{functional ideal I} Functional ideals play  the same role in approach spaces as what filters do in   topological spaces, see, e.g. \cite{CLR,CVO,RL97,Lowen15,LOV2008,LV2008}. This example shows that   functional ideals are essentially  bounded saturated prefilters for the product t-norm $\&_P$. For each $X$, let $BX$ denote the set of all bounded functions $X\lra[0,\infty]$.
Then, a \emph{functional ideal}  on $X$ in the sense of \cite{Lowen15,LOV2008,LV2008} is a subset $\mathfrak{I}$ of $BX$ subject to the following conditions:
\begin{enumerate}[label=\rm(\roman*)] \item If $\lam\in \mathfrak{I}$ and $\mu\leq \lam$ (pointwise)  then $\mu\in \mathfrak{I}$. \item If $\lam,\mu\in \mathfrak{I}$ then there is some $\gamma\in \mathfrak{I}$ such that $\gamma(x)\geq\max\{\lam(x),\mu(x)\}$ for all $x\in X$. \item $\mathfrak{I}$ is saturated in the sense that for each $\lam:X\lra[0,\infty]$: \[(\forall\epsilon>0, \exists\thinspace\mu\in \mathfrak{I}, \lam\leq\mu+\epsilon) \Rightarrow \lam\in \mathfrak{I}.\] \end{enumerate}

Since   $x\mapsto \mathrm{e}^{-x}$ is an isomorphism between Lawvere's quantale $([0,\infty]^{\rm op},+,0)$ and the quantale $([0,1], \with_P,1)$,    a functional ideal on a set $X$  is essentially a bounded saturated prefilter on $X$ (with respect to the   t-norm $\with_P$).\end{exmp}

For a saturated prefilter $F$  on a set $X$,  let $\varrho_X(F)$ be the set  of  bounded elements in $F$. Then $\varrho_X(F)$ is a bounded  saturated prefilter on $X$,    the largest one contained in $F$, called the \emph{bounded coreflection} of $F$. Moreover,    \[\varrho_X(F)=\{\lam\vee \epsilon_X\mid \lam\in F, \epsilon>0\}.\]

Given a map $f:X\lra Y$ and a bounded saturated prefilter $F$   on $X$, let \[f_B(F)=\varrho_Y(f(F))=\{\mu\in[0,1]^Y\mid \mu~\text{is bounded and}~\mu\circ f\in F\}.\]   Then the assignment \[X\to^f Y~\mapsto~ \BSF(X)\to^{f_B}\BSF(Y)\] defines a functor \[\BSF:{\sf Set}\to{\sf Set}.\] Moreover, $\varrho=\{\varrho_X\}_X$ is a natural transformation $\SPF\lra\BSF$ and it is an epimorphism in the category of endofunctors on {\sf Set}.

We would like to warn the reader that though $\BSF(X)$ is a subset of $\SPF(X)$ for every set $X$, the functor $\BSF$ is, in general, not a subfunctor of $\SPF$.

Now we define two natural transformations $\widetilde{\mathfrak{d}}:{\rm id}\lra\BSF$ and   $\widetilde{\mathfrak{n}}:\BSF^2\lra\BSF$.

For each $x\in X$, let $\widetilde{\mathfrak{d}}_X(x)$ be the bounded coreflection of $\mathfrak{d}_X(x)$, i.e., \[\widetilde{\mathfrak{d}}_X(x)=\{\lam\in[0,1]^X\mid  \lam(x)=1, \lam~\text{is bounded}\}.\] Then, $\widetilde{\mathfrak{d}}=\{\widetilde{\mathfrak{d}}_X\}_X$ is  a natural transformation  ${\rm id}\lra\BSF$, it is indeed the composite of $\varrho: \SPF\lra\BSF$ and $\mathfrak{d}:\id\lra\SPF$.

For each bounded saturated prefilter $\CF$ on $\BSF(X)$, let \begin{equation*}
\widetilde{\mathfrak{n}}_X(\CF)= \{\lam\in[0,1]^X\mid \lam~\text{is bounded}, \widetilde{\lam} \in\CF\},\end{equation*} where \[\widetilde{\lam}(F)=\bv_{\mu\in F}\sub_X(\mu,\lam) \] for each bounded saturated prefilter $F$ on $X$. Put differently, $\widetilde{\mathfrak{n}}_X(\CF)$ is the bounded coreflection of $\mathfrak{n}_X(k_X(\CF))$, where $k_X$ refers to the inclusion of $\BSF(X)$ in $\SPF(X)$. Then, \begin{equation}\label{tilden}\widetilde{\mathfrak{n}}=\{ \widetilde{\mathfrak{n}}_X\}_X\end{equation} is a natural transformation   $\BSF^2\lra\BSF$ (see Propositions \ref{n is a natural trans} and \ref{naturality of d and n} below).

The question of this section is:
\begin{ques}When is the triple \[(\BSF, \widetilde{\mathfrak{n}}, \widetilde{\mathfrak{d}})\]  a monad?\end{ques}
As in previous sections,  the  idea to answer this question  is to relate bounded saturated prefilters to certain kind of $\sQ$-semifilters, conical bounded $\sQ$-semifilters in this case.

\begin{exmp}[Functional ideals, II] \label{functional ideal II}
Let  $\&$ be the product t-norm. Then, the natural transformation $\widetilde{\mathfrak{n}}$ in  (\ref{tilden}) is essentially the multiplication  of the functional ideal monad   in \cite[Subsection 2.3]{CVO}.
To see this, first we  show that for each bounded saturated prefilter $F$ on $X$ and each $\alpha>0$, \[\alpha\otimes F=\{\lam\in[0,1]^X, \exists\mu\in F, \alpha\with\mu\leq\lam\}\] is a bounded saturated prefilter. It suffices to check that $\alpha\otimes F$ is saturated. Actually, \begin{align*}\bv_{\gamma\in\alpha\otimes F}\sub_X(\gamma,\lam)=1&\Longrightarrow  \bv_{\mu\in  F}\sub_X(\alpha\with\mu,\lam)=1\\ &\Longrightarrow\bv_{\mu\in  F}\sub_X(\mu,\alpha\ra\lam)=1\\ &\Longrightarrow \alpha\ra\lam\in F \\ & \Longrightarrow \lam\in \alpha\otimes F.\quad~~(\alpha\with(\alpha\ra\lam)\leq\lam)\end{align*}

Next, let $\mu:X\lra[0,1]$ be a bounded map, $p\in[0,1]$, and let $\CF$ be  a bounded saturated prefilter  on $\BSF(X)$. Then, for each   $F\in\BSF(X)$,   \begin{align*}p\leq \widetilde{\mu}(F)&\iff p\leq \bv_{\lam\in F}\sub_X(\lam,\mu)\\ &\iff p\leq \bv\{\alpha\in(0,1]\mid \exists \lam\in F, \alpha\with\lam\leq\mu\}\\ &\iff  p\leq \bv\{\alpha\in(0,1]\mid \mu\in \alpha\otimes F\}.\end{align*}

Therefore, $\widetilde{\mu}$ is essentially the map $l_\mu$ in \cite[Subsection 2.2]{CVO}, and consequently, the natural transformation  $\widetilde{\mathfrak{n}}$  is the multiplication  of the  functional ideal  monad  in \cite[Subsection 2.3]{CVO}.
 \end{exmp}

A  $\sQ$-semifilter $\mathfrak{F}$ on a set $X$ is said to be \emph{bounded} if $\mathfrak{F}(\mu)<1$ whenever $\mu:X\lra[0,1]$ is unbounded.

\begin{lem}\label{bounded Q-filter} Let $\mathfrak{F}$ be a  $\sQ$-semifilter on a set $X$. If $\mathfrak{F}$ is bounded, then the saturated prefilter $\Gamma_X(\mathfrak{F})$ is bounded. The converse implication also holds when $\mathfrak{F}$ is conical. Therefore, the conical coreflection of a bounded $\sQ$-semifilter is bounded.
\end{lem}
\begin{proof}
If $\mathfrak{F}$ is bounded and $\mu\in \Gamma_X(\mathfrak{F})$, then $\mathfrak{F}(\mu)=1$, hence $\mu$ is bounded. As for the converse implication, assume that $\mathfrak{F}$ is conical. Since $\Gamma_X(\mathfrak{F})$ is bounded and \[\mathfrak{F}(\mu)=\bv_{ \lam\in \Gamma_X(\mathfrak{F})}\sub_X(\lam,\mu),\] it suffices to check that if $\mu$ is unbounded, then there is some $b<1$ such that $\sub_X(p_X,\mu)\leq b$ for all $p>0$. Take for $b$   an idempotent element of $\&$ in $ (0,1)$ if  $\&$ has one, otherwise, take an arbitrary element in $(0,1)$ for $b$. Since $\mu$ is unbounded, for each $p>0$ there is some $z\in X$ such that $\mu(z)<b\wedge p$. Then \[\sub_X(p_X,\mu)\leq p\ra\mu(z)\leq b,\] as desired.
\end{proof}

For each set $X$, let \[ \BCF(X) \] be the set of conical bounded $\mathsf{Q}$-semifilters on $X$.

For each $\sQ$-semifilter $\mathfrak{F}$ on  $X$,    \[ \vartheta_X(\mathfrak{F})\coloneqq\Lambda_X\circ\varrho_X\circ\Gamma_X( \mathfrak{F} )  \] is clearly  the largest conical bounded $\sQ$-semifilter that is smaller than or equal to $\mathfrak{F}$, and is called the \emph{conical bounded coreflection} of $\mathfrak{F}$.

For each map $f:X\lra Y$ and each conical bounded $\sQ$-semifilter $\mathfrak{F}$ on $X$, let   $f_B(\mathfrak{F})$ be the conical bounded coreflection of $f(\mathfrak{F})$, i.e., \[f_B(\mathfrak{F})=\vartheta_Y(f(\mathfrak{F})). \]  Then we obtain a functor \[\BCF:{\sf Set}\to{\sf Set}, \quad f\mapsto f_B. \]

We hasten to note that though $\BCF(X)$ is a subset of $\SFQ(X)$ for each set $X$,   the functor $\BCF$  is, in general, not a subfunctor of $\SFQ$. But,   \[ \vartheta=\{\vartheta_X\}_X \]
is a natural transformation from \(\SFQ\) to \(\BCF\) and it is  an epimorphism in the category of endofunctors on {\sf Set}.

The following lemma is a useful property of the map $f_B$.

\begin{lem}\label{image of CBF} Let $f:X\lra Y$ be a map. Then, for each conical  $\sQ$-semifilter $\mathfrak{F}$ and each bounded $\lam\in [0,1]^Y$, $f_B(\mathfrak{F})(\lam)=f(\mathfrak{F})(\lam)$. \end{lem}

\begin{proof}
Since $\lam$ is bounded, there is some $\epsilon>0$ such that $\lam\geq\epsilon_Y$. Since $ f(\mathfrak{F}) $ is conical, then \begin{align*}
f(\mathfrak{F})(\lambda)&=\bigvee\{\sub_Y(\mu,\lambda)\mid f(\mathfrak{F})(\mu)=1\}   \\
&= \bigvee\{\sub_Y(\mu\vee\epsilon_Y,\lambda)\mid f(\mathfrak{F})(\mu)=1\}  \\
&= \bigvee\{\sub_Y(\mu,\lambda)\mid f(\mathfrak{F})(\mu)=1, \mu~\text{is bounded}\} \\&=f_B(\mathfrak{F})(\lambda).
\qedhere\end{align*} \end{proof}

It is clear that the correspondence \(F\mapsto\Lambda_X(F)\)    restricts to a bijection between bounded saturated prefilters and conical bounded $\sQ$-semifilters,  so, $\BCF$ is naturally isomorphic to $\BSF$: $$\bfig
\morphism(550,0)|b|/{@{>}@<3pt>}/<-550,0>[\BCF.`\BSF;\Gamma]
\morphism(0,0)|a|/{@{>}@<3pt>}/<550,0>[\BSF`\BCF.;\Lambda]
\efig$$ Thus, the question considered in this section is equivalent to whether the functor $\BCF$ gives rise to a monad.

For each $x\in X$, let $\widetilde{\sfd}_X(x)$ be the conical bounded coreflection of $\sfe_X(x)$, i.e., \[\widetilde{\sfd}_X(x)=\vartheta_X(\sfe_X(x)).\] Then $\widetilde{\sfd}=\{\widetilde{\sfd}_X\}_X$ is  a natural transformation   ${\rm id}\lra\BCF$.

\begin{lem}\label{bounded filter closed under diagonal} If $\mathbb{F}$ is a bounded  $\sQ$-semifilter on   $\BCF(X)$, then the diagonal $\sQ$-semifilter \(\sfm_X(j_X(\mathbb{F}))\) is also bounded, where $j_X$ denotes the inclusion $\BCF(X)\lra\SFQ(X)$. \end{lem}

Before proving the conclusion, we remind the reader  that $\{j_X\}_X$ is, in general,  not a natural transformation from $\BCF$ to $\SFQ$.

\begin{proof}We  check that if $\mu\in[0,1]^X$ is unbounded,   then $\sfm_X(j_X(\mathbb{F}))(\mu)<1$.  Since $\mathbb{F}$ is bounded and \(\sfm_X(j_X(\mathbb{F}))(\mu) =\mathbb{F}(\widehat{\mu}\circ j_X)\),  it suffices to check that $\widehat{\mu}\circ j_X$ is unbounded. This follows from that \begin{align*}\bw_{\mathfrak{G}\in\BCF(X)}\widehat{\mu}\circ j_X(\mathfrak{G}) &= \bw_{\mathfrak{G}\in\BCF(X)}\mathfrak{G}(\mu)\leq\bw_{x\in X}\widetilde{\sfd}_X(x)(\mu)\leq \bw_{x\in X}\mu(x)=0. \qedhere\end{align*} \end{proof}

\begin{prop}\label{n is a natural trans} For each set $X$, define \[\widetilde{\sfn}_X:\BCF^2(X)\lra\BCF(X)\] by letting $\widetilde{\sfn}_X(\mathbb{F})$ be the conical coreflection of the diagonal $\sQ$-semifilter $\sfm_X(j_X(\mathbb{F}))$. Then   $\widetilde{\sfn}= \{\widetilde{\sfn}_X\}_X$ is a natural transformation   $\BCF^2\lra\BCF$. \end{prop}

\begin{proof}  We show that for each map $f:X\lra Y$, the following diagram is commutative: \[\bfig\square<850,450>[\BCF^2(X)`\BCF(X)`\BCF^2(Y)`\BCF(Y);\widetilde{\sfn}_X` (f_B)_B` f_B` \widetilde{\sfn}_Y]\efig\]

Let $\mathbb{F}$ be a conical bounded $\sQ$-semifilter on $\BCF(X)$. Since both $f_B\circ \widetilde{\sfn}_X(\mathbb{F})$ and $\widetilde{\sfn}_Y\circ (f_B)_B(\mathbb{F})$ are conical, it suffices to check that  for every $\lam\in[0,1]^Y$, \[f_B\circ \widetilde{\sfn}_X(\mathbb{F})(\lam)=1\iff \widetilde{\sfn}_Y\circ (f_B)_B(\mathbb{F}) (\lam)=1.\]

On one hand,  \begin{align*}f_B\circ \widetilde{\sfn}_X(\mathbb{F})(\lam)=1&\iff  \lam\in\Gamma_Y(f_B\circ \widetilde{\sfn}_X(\mathbb{F})) \\
&\iff \lam~\text{is bounded and}~ f(\widetilde{\sfn}_X(\mathbb{F}))(\lam) =1 \\ &\iff \lam~\text{is bounded and}~ \widetilde{\sfn}_X(\mathbb{F}) (\lam\circ f) =1 \\ &\iff \lam~\text{is bounded and}~ \sfm_X(j_X(\mathbb{F})) (\lam\circ f) =1\\ &\iff \lam~\text{is bounded and}~ \mathbb{F}  (\widehat{\lam\circ f}\circ j_X) =1 \\ &\iff \lam~\text{is bounded and}~ \mathbb{F}  (\mathfrak{F}\mapsto \mathfrak{F}(\lam\circ f)) =1.\end{align*}

On the other hand, since $\sfm_Y(j_Y((f_B)_B(\mathbb{F})))$ is bounded (Lemma \ref{bounded filter closed under diagonal}) and   \[f_B(\mathfrak{F})(\lam)=f(\mathfrak{F})(\lam)\] for each   $\mathfrak{F}\in\BCF(X)$ and each bounded $\lam\in[0,1]^Y$ (Lemma \ref{image of CBF}), we have
\begin{align*}\widetilde{\sfn}_Y\circ (f_B)_B(\mathbb{F}) (\lam)=1 & \iff \sfm_Y(j_Y((f_B)_B(\mathbb{F}))) (\lam)=1 \\ & \iff   j_Y((f_B)_B(\mathbb{F}))  (\widehat{\lam})=1 \\ & \iff   (f_B)_B(\mathbb{F})   (\widehat{\lam}\circ j_Y)=1 \\ & \iff\widehat{\lam}\circ j_Y~\text{is bounded and}~ \mathbb{F}    (\widehat{\lam}\circ j_Y\circ f_B)=1 \\ & \iff \lam~\text{is bounded and}~ \mathbb{F}(\mathfrak{F}\mapsto f_B(\mathfrak{F})(\lam))=1 \\ & \iff \lam~\text{is bounded and}~\mathbb{F}(\mathfrak{F}\mapsto  \mathfrak{F} (\lam\circ f))=1.  \end{align*}

The proof is completed. \end{proof}

\begin{prop}\label{naturality of d and n} $\widetilde{\mathfrak{d}}=\Gamma\circ \widetilde{\sfd}$ and $\widetilde{\mathfrak{n}}=\Gamma\circ \widetilde{\sfn}\circ (\Lambda*\Lambda)$. \end{prop}

\begin{proof}The verification is similar to that for Proposition \ref{natural of d and n}, and is included here for convenience of the reader. The first equality is obvious. As for the second, suppose that  $\CF$ is a bounded saturated prefilter on $\BSF(X)$. Let $j_X$ be the inclusion \[\BCF(X)\lra\SFQ(X).\]  Then, for each $\lam\in\sQ^X$,   \begin{align*}\lam\in\Gamma_X\circ\widetilde{\sfn}_X\circ (\Lambda*\Lambda)_X(\CF)
&\iff \sfm_X(j_X ((\Lambda*\Lambda)_X(\CF)))(\lam)=1\\
&\iff    (\Lambda*\Lambda)_X(\CF)(\widehat{\lam}\circ j_X)=1 \\ &\iff \Lambda_{\BSF(X)}(\CF)(\widehat{\lam}\circ j_X\circ\Lambda_X)=1\\ &\iff \widehat{\lam}\circ j_X\circ\Lambda_X\in\CF.  \quad~~(\CF~\text{is saturated}) \end{align*} Since for each bounded saturated prefilter $F$ on $X$, we have  \[ \widehat{\lam}\circ j_X\circ\Lambda_X(F)=\Lambda_X(F)(\lam)=\bv_{\mu\in F}\sub_X(\mu,\lam) =\widetilde{\lam}(F), \] it follows that $\widetilde{\mathfrak{n}}=\Gamma\circ \widetilde{\sfn}\circ (\Lambda*\Lambda)$. \end{proof}

\begin{thm}\label{main3} Let   $\&$ be a continuous t-norm that is not   isomorphic to  the {\L}ukasiewicz t-norm and let $\sQ=([0,1],\with,1)$. Then  the following statements are equivalent: \begin{enumerate}[label=\rm(\arabic*)]  \item The t-norm $\&$ satisfies the condition (S).
 \item  The triple $(\BCF, \widetilde{\sfn},\widetilde{\sfd})$ is a monad.
\item  The triple $(\BSF, \widetilde{\mathfrak{n}}, \widetilde{\mathfrak{d}})$ is a monad.  \end{enumerate}
     \end{thm}

Before proving the theorem, we would like to point out that  Proposition 2.9 in \cite{CVO} together with Example \ref{functional ideal II} already imply that $(\BSF, \widetilde{\mathfrak{n}}, \widetilde{\mathfrak{d}})$ is a monad when $\&$ is isomorphic to the product t-norm.

\begin{proof}
$(1)\Rightarrow(2)$   Since   $\&$ satisfies the condition (S), then for each conical bounded $\sQ$-semifilter $\mathbb{F}$ on $\BCF(X)$, the diagonal $\sQ$-semifilter $\sfm_X(j_X(\mathbb{F}))$ is conical, hence \[\widetilde{\sfn}_X(\mathbb{F})=\sfm_X(j_X(\mathbb{F})).\] In other words,   conical bounded $\sQ$-semifilters are ``closed under multiplication''. With   help of this fact and that two bounded   conical $\mathsf{Q}$-semifilters are equal if and only if they are equal on bounded elements, it is routine to check that $(\BCF,\widetilde{\sfn}, \widetilde{\sfd})$ is a monad.

$(2)\Rightarrow(1)$ In the proof of $(3)\Rightarrow(1)$  in Theorem \ref{main1}, replace $\gamma$ by \[\epsilon_X\vee p(1-x)\] for some $0<\epsilon<p$ and replace $\mathfrak{G}$ by the conical bounded $\sQ$-filter  on $X$ generated by  \[\{  1_{A_n}\vee\delta_X \mid n\geq 1, \delta>0  \} ,\] where $A_n=\{ 1/m \mid m\geq n\}$.

$(2)\Leftrightarrow(3)$ Proposition \ref{naturality of d and n}. \end{proof}
\begin{rem}
If $\&$ is a continuous t-norm that  is isomorphic to  the {\L}ukasiewicz t-norm, then  every saturated prefilter is bounded by Proposition \ref{sat is bounded}. For such a t-norm,  we simply let   $(\BSF, \widetilde{\mathfrak{n}}, \widetilde{\mathfrak{d}})=(\SPF, \mathfrak{n},\mathfrak{d})$  and  $ (\BCF, \widetilde{\sfn},\widetilde{\sfd})= (\CSFQ, \sfn,\sfd)$.  Thus agreed, Theorem \ref{main3} holds for every continuous t-norm.
\end{rem}
\section{Summary}
 Let $\sQ=([0,1],\&,1)$ with $\&$ being a continuous t-norm.
It is proved that  the following statements are equivalent: \begin{enumerate}[label=\rm(\arabic*)]  \item The  implication operator of $\&$ is continuous at each point off the diagonal.
\item The conical $\sQ$-semifilter functor  $\CSFQ$   is a submonad of $(\SFQ,\sfm,\sfe)$.
\item The conical $\sQ$-filter functor  $\CFQ$   is a submonad of $(\SFQ,\sfm,\sfe)$. \item The triple   $(\SPF, \mathfrak{n}, \mathfrak{d})$ is a monad.
\item  The triple $(\topF, \mathfrak{n},\mathfrak{d})$ is a monad.
 \item The triple  $(\BSF, \widetilde{\mathfrak{n}}, \widetilde{\mathfrak{d}})$ is a monad.
         \end{enumerate} 

The relations among these monads are summarized in the following diagram:
\[\bfig \morphism(0,0)|r|/@{->}@<2.5pt>/<0,-450>[\SFQ`\CSFQ;\mathfrak{c}]
\morphism(0,0)|l|/@{<-}@<-2.5pt>/<0,-450>[\SFQ`\CSFQ;\mathfrak{i}] \morphism(0,-450)|a|/@{->}@<0pt>/<600,0>[\CSFQ`\BCF;\vartheta]
\morphism(-600,0)|a|/@{->}@<0pt>/<600,0>[\sQ\text{-}{\sf Fil}`\SFQ;\mathfrak{i}] \morphism(-600,-450)|l|/@{->}@<2.5pt>/<0,450>[\CFQ`\sQ\text{-}{\sf Fil};\mathfrak{i}] \morphism(-600,-450)|r|/@{<-}@<-2.5pt>/<0,450>[\CFQ`\sQ\text{-}{\sf Fil};\mathfrak{c}]
 \morphism(-600,-450)|a|/@{->}@<0pt>/<600,0>[\CFQ`\CSFQ;\mathfrak{i}]
 \morphism(0,-450)|r|/@{->}@<2.5pt>/<0,-450>[\CSFQ`\SPF;\Gamma] \morphism(0,-450)|l|/@{<-}@<-2.5pt>/<0,-450>[\CSFQ`\SPF;\Lambda] \morphism(0,-900)|a|/@{->}@<0pt>/<600,0>[\SPF`\BSF;\varrho]
 \morphism(600,-450)|r|/@{->}@<2.5pt>/<0,-450>[\BCF`\BSF;\Gamma] \morphism(600,-450)|l|/@{<-}@<-2.5pt>/<0,-450>[\BCF`\BSF;\Lambda] \morphism(-600,-900)|r|/@{<-}@<-2.5pt>/<0,450>[\topF`\CFQ;\Gamma]
 \morphism(-600,-900)|l|/@{->}@<2.5pt>/<0,450>[\topF`\CFQ;\Lambda] \morphism(-600,-900)|a|/@{->}@<0pt>/<600,0>[\topF`\SPF;\mathfrak{i}]
   \efig\]
\begin{itemize}
\item each arrow is  a  monad morphism; \item each  monad in the diagram is  order-enriched; \item each $\mathfrak{i}$ is a monomorphism; \item   $\vartheta$ and $\varrho$ are   epimorphisms; \item  each $\Lambda$ is an isomorphism with  inverse given by $\Gamma$; \item   $\mathfrak{c}\circ \mathfrak{i}$ is the identity. \end{itemize}

These monads are  useful in  \emph{monoidal topology} \cite{Gahler1992,Monoidal top} and in the theory of quantale-enriched orders. As an example, since the quantale $\sQ=([0,1], \with_P,1)$  is isomorphic to Lawvere's quantale $([0,\infty]^{\rm op},+,0)$, the Kleisli monoids and Eilenberg-Moore algebras of the monad $(\BSF, \widetilde{\mathfrak{n}}, \widetilde{\mathfrak{d}})$ are essentially  approach spaces \cite{CVO} and   injective approach spaces  \cite{GH}, respectively.
As another example,  let $\sQ=([0,1], \with,1)$ with $\&$ being a continuous t-norm that satisfies the condition (S). Then,   Kleisli monoids  and Eilenberg-Moore algebras of the monad $(\SPF, \mathfrak{n},\mathfrak{d})$  are CNS spaces \cite{LZ2018} and complete and continuous $\sQ$-categories \cite{LZ2020,LiZ18b}, respectively.

\end{document}